\documentclass[reqno,12pt]{amsart}

\usepackage{array,multirow}
\usepackage{times}
\usepackage{amsmath,amsbsy,amsfonts,amssymb}
\usepackage{graphicx}
\usepackage{pstricks}

\usepackage{hyperref}

\usepackage{bm}
\usepackage{dsfont}
\usepackage{mathrsfs}
\usepackage{multicol}        
\usepackage[displaymath,pagewise]{lineno} 
\usepackage{verbatim}
\usepackage{subfigure}
\usepackage{color}

\newcommand{\mcal}{\mathcal}
\newcommand{\mbb}{\mathbb}
\newtheorem{theorem}{Theorem}[section]

\newtheorem{lemma}[theorem]{Lemma}
\newtheorem{corollary}[theorem]{Corollary}

\theoremstyle{definition}
\newtheorem{definition}[theorem]{Definition}

\theoremstyle{remark}
\newtheorem{remark}[theorem]{Remark}

  \newcounter{mnote}
  \let\oldmarginpar\marginpar
    \renewcommand\marginpar[1]{\-\oldmarginpar[\raggedleft\footnotesize #1]%
    {\raggedright\footnotesize #1}}

\numberwithin{equation}{section}  
\numberwithin{figure}{section}
\numberwithin{table}{section}

\usepackage{mathrsfs}
\DeclareGraphicsExtensions{.pdf,.jpg,.png,.ps,.eps,.gif}

\newcommand{\dx}{\,{\rm d}x}

\newcommand{\R}{\mathcal{R}}
\newcommand{\V}{\mathcal{V}}

\newcommand{\B}{\mathcal{B}}

\newcommand{\E}{\mathcal{E}}
\newcommand{\F}{\mathcal{F}}
\newcommand{\I}{\mathcal{I}}

\newcommand{\N}{\mathcal{N}}
\newcommand{\T}{\mathcal{T}}
\renewcommand{\P}{\mathcal{P}}

\newcommand{\hf}{\frac{1}{2}}

\begin{document}
\title[Local Multilevel Preconditioners for Jump Coefficients Problems]
{Local Multilevel Preconditioners for Elliptic Equations \\with Jump
Coefficients on Bisection Grids}
\author[L. Chen]{Long Chen}
\email{chenlong@math.uci.edu}
\address{Department of Mathematics, University of California at Irvine, CA 92697, USA. }
\author[M. Holst]{Michael Holst}
\email{mholst@math.ucsd.edu}
\address{
Department of Mathematics, University of California at San Diego, CA 92093, USA. 
}
\author[J. Xu]{Jinchao Xu}
\email{xu@math.psu.edu}
\address{Department of Mathematics, 
Pennsylvania State University, University Park, PA 16802, USA. 
 }
\author[Y. Zhu]{Yunrong Zhu}
\email{zhu@math.ucsd.edu}
\address{
Department of Mathematics,
University of California at San Diego, CA 92093, USA. }

\begin{abstract}
  The goal of this paper is to design optimal multilevel solvers for
  the finite element approximation of second order linear elliptic
  problems with piecewise constant coefficients on bisection grids. Local multigrid and BPX preconditioners  are constructed based on local smoothing only at the newest vertices and their immediate neighbors. The analysis of eigenvalue distributions for these local multilevel preconditioned systems shows that  there are only a fixed number of eigenvalues which are deteriorated by the large jump. The remaining eigenvalues are bounded uniformly with respect to the coefficients and the meshsize. Therefore, the resulting preconditioned conjugate gradient algorithm will converge with an asymptotic rate independent of the coefficients and logarithmically with respect to the meshsize. As a result, the overall computational complexity is nearly optimal.
\end{abstract}

\maketitle

\section{Introduction}

In this article, we construct robust multilevel preconditioners for the finite element discretization of second order linear elliptic equations with strongly discontinuous coefficients. We extend corresponding results on uniform grids~\cite{Xu.J;Zhu.Y2008} to locally refined grids obtained by bisection methods. We consider the following model problem :
\begin{equation}\label{eq:model}
\left\{\begin{array}{lll}
      &&-\nabla\cdot(a\nabla u)=f \mbox{ in } \Omega,\\
      &&u=g_D \mbox{ on } \Gamma_D, \;\;
      a\frac{\partial u}{\partial n}=g_N \mbox{ on } \Gamma_N
 \end{array}\right.
\end{equation}
where $\Omega\in \mathbb R^d$ is a polygon (for $d=2$) or
polyhedron (for $d=3$) with Dirichlet boundary $\Gamma_D$ and Neumann
boundary $\Gamma_N$ such that $\Gamma _D \cup \Gamma _N=\partial \Omega$. The diffusion coefficient $a=a(x)$ is piecewise constant. More precisely, the domain $\Omega$ is partitioned into $M$ open disjoint polygonal or polyhedral regions $\Omega_i\; (i=1,\cdots, M)$ and
$$a|_{\Omega_i}=a_i, \;\; i=1,\dots, M$$
where each $a_i$ is a positive constant. The regions $\Omega_{i}\; (i=1,\cdots M)$ may possibly have complicated geometry but we assume that they are completely resolved by an initial triangulation $\T_{0}.$ Our analysis can be carried
through to more general cases when $a(x)$ varies moderately in each
subdomain and to other types of boundary conditions in a straightforward way.

The problem \eqref{eq:model} belongs to the class of interface problems or transmission problems, which are relevant to many applications such as groundwater
flow \cite{Kees.C;Miller.C;Jenkins.E;Kelley.C2003},  electromagnetics
\cite{Heise.B;Kuhn.M1996}, semiconductor modeling
\cite{Coomer.R;Graham.I1996,Meza.J;Tuminaro.R1996}, and fuelcells
\cite{Wang.Z;Wang.C;Chen.K2001}. The coefficients in these applications may have large jumps across interfaces between regions with different material properties, i.e. $J(a) : = \max _i a_i / \min _i a_i \gg 1.$ Due to $J(a)$ and the mesh size, the finite element discretization of \eqref{eq:model} is usually very ill-conditioned, which leads to deterioration in the rate of convergence of multilevel and domain decomposition methods
\cite{Alcouffe.R;Brandt.A;Dendy.J;Painter.J1981,Graham.I;Hagger.M1999,Vuik.C;Segal.A;Meijerink.J1999}. 

In some special situations, one is able to show the (nearly) uniform convergence of the multilevel and (overlapping) domain decomposition methods (see \cite{Bramble.J;Xu.J1991,Wang.J1994,Wang.J;Xie.R1994,Dryja.M;Sarkis.M;Widlund.O1996,Oswald.P1999c} for examples). For general cases, one usually need some special techniques to obtain robust iterative methods, (cf. \cite{Chan.T;Wan.W2000,Scheichl.R;Vainikko.E2007,Graham.I;Lechner.P;Scheichl.R2007,Aksoylu.B;Graham.I;Klie.H;Scheichl.R2008}). Recently in~\cite{Xu.J;Zhu.Y2008,Zhu.Y2008}, we analyzed the eigenvalue distributions of the standard multilevel and overlapping domain decomposition preconditioned systems, and showed that there are only a small fixed number of eigenvalues that may deteriorate due 
to the discontinuous jump or mesh size, and that all the other eigenvalues are 
bounded below and above nearly uniformly with respect to the jump and mesh size. As a result, we proved that the convergence rate of the preconditioned conjugate gradient 
method is uniform with respect to the large jump, and depends logarithmically on mesh size. These results ensure that the standard multilevel and domain decomposition preconditioners are efficient and robust for finite element discretization of \eqref{eq:model} on {\it quasi-uniform grids}. In this paper, we extend our results to {\it locally refined grids}.

The discontinuity of diffusion coefficients causes a lack of regularity of the solution to \eqref{eq:model}, which in turn, leads to deterioration in the rate of convergence for finite element approximations over quasi-uniform triangulations. Adaptive finite element methods through local mesh refinement can be applied to recover the optimal rate of convergence~\cite{Cascon.J;Kreuzer.C;Nochetto.R;Siebert.K2008}. 
In order to achieve optimal computational complexity in adaptive finite element methods, it is imperative to design fast algorithms for solving the linear system of equations arising from the finite element discretization. The distinct feature of applying multigrid methods on locally refined meshes is that the number of nodes of nested meshes obtained by local refinements may not grow exponentially, violating one of the key properties of multilevel methods on uniform meshes that leads to optimal $\mcal O(N)$ complexity. Indeed, let $N$ be the number of unknowns in the finest space, the complexity of multilevel methods with global smoothers can be as bad as $\mcal O(N^2)$~\cite{Mitchell.W1992}. This prevents direct application of algorithms and theories developed in~\cite{Xu.J;Zhu.Y2008} for quasi-uniform grids to locally refined grids.

To achieve optimal $\mcal O(N)$ complexity, the smoothing step in each level must be restricted to the newly added unknowns and their neighbors (see \cite{Bai.D;Brandt.A1987,Bramble.J;Pasciak.J;Xu.J1990,Mitchell.W1992}). Such methods are referred to as {\it local multilevel methods} in~\cite{Bai.D;Brandt.A1987}. As an extreme case, one can preform the smoothing only on newly added nodes turning a coarse grid to a fine grid. The resulting method is known as the hierarchical basis method~\cite{Yserentant.H1990,Bank.R;Dupont.T;Yserentant.H1988}. In two dimensions, hierarchical basis methods are proven to be robust for jump coefficient problems on locally refined meshes (cf.~\cite{Bank.R;Dupont.T;Yserentant.H1988}). In three dimensions, however, classic multilevel and domain decomposition methods, including the hierarchical basis multigrid methods, deteriorate rapidly due to the presence of discontinuity of coefficients. To obtain robust rates of convergence for multigrid methods, one has to use special coarse spaces~\cite{Dryja.M;Sarkis.M;Widlund.O1996,Sarkis.M1997} or assume that the distribution of diffusion coefficients satisfies the so called quasi-monotone condition~\cite{Dryja.M;Sarkis.M;Widlund.O1996}. Therefore the three dimensional case is much more difficult. There are other works~\cite{Aksoylu.B;Holst.M2006,Hiptmair.R;Zheng.W2009a} on optimal complexity of local multilevel methods in three dimensions, but the problems with discontinuous coefficients remain open. 

In this article, we shall design and prove the efficiency and robustness of local multilevel preconditioners for the finite element discretization of problem \eqref{eq:model} on bisection grids -- one class of locally refined grids.  In these preconditioners, we use a global smoothing in the finest mesh; and for each newly added node, we perform smoothing only for three vertices - the new vertex and its two parents vertices (the vertices sharing the same edge with the new vertex). We analyze the eigenvalue distribution of the multilevel preconditioned matrix, and prove that there are only a fixed number of small eigenvalues deteriorated by the coefficient and mesh-size; the other eigenvalues are bounded nearly uniformly. Thus, the resulting preconditioned conjugate gradient algorithm converges uniformly with respect to the jump and logarithmically with respect to the mesh size of the discretization.  We establish our results of this type in both two and three dimensions. 

To emploit the geometric structure of bisection grids, we use the decomposition of bisection grids developed in the recent work \cite{Chen.L;Nochetto.R;Xu.J2007,Xu.J;Chen.L;Nochetto.R2009}. This approach enables us to introduce a natural decomposition of the finite element space into subspaces consisting only the newest vertices and their two parents vertices. In the analysis of these local multilevel preconditioners, one of the key ingredient is the \emph{stable decomposition} (see Theorem~\ref{th:stda}). For the standard multilevel preconditioners on uniform mesh, in \cite{Xu.J;Zhu.Y2008} we used the approximation and stability properties of the weighted $L^{2}$ projection (cf. \cite{Bramble.J;Xu.J1991}) to construct a stable decomposition. This weighted $L^{2}$ projection is no longer applicable for the local multilevel preconditioners, since it is a global projection. In order to preserve the local natural of the highly graded meshes, we introduce a local interpolation operator, which we manage to prove similar approximation and stability properties (see Theorem~ \ref{thm:wl2app} and \ref{thm:wl2}) as the weighted $L^{2}$-projection. Our local quasi-interpolation operator and the corresponding analysis is more delicate than that in~\cite{Chen.L;Nochetto.R;Xu.J2007,Xu.J;Chen.L;Nochetto.R2009} for the Poisson equation.  
We should remark that due to this space decomposition, we are able to remove the assumption, {\it nested local refinement}, which is used in most existing work on multilevel methods on local refinement grids~\cite{Aksoylu.B;Holst.M2006,Hiptmair.R;Zheng.W2009a}. 

The rest of the paper is organized as follows. In Section~\ref{sec:pre}, we give
some notation and recall some fundamental results as in
\cite{Xu.J;Zhu.Y2008}. In Section~\ref{sec:decomp}, we study bisection grids,
and review some technical tools from
\cite{Chen.L;Nochetto.R;Xu.J2007,Xu.J;Chen.L;Nochetto.R2009}. Here we restrict ourself to a
kind of special bisection scheme, namely the newest vertex bisection. Then
in Section~\ref{sec:decomp}, we study some technical results of space
decomposition, and present the optimal/stable decomposition and the
strengthened Cauchy-Schwarz inequality on bisection grids. In Section~\ref{sec:precond}, we analyze multilevel preconditioners, i.e., the BPX
preconditioner and the multigrid $V$-cycle preconditioner, and prove convergence results for the preconditioned conjugate gradient
algorithm. In Section~\ref{sec:num}, we present numerical experiments to support our theoretical results. 

Throughout the article, we will use the following short notation, $x\lesssim y$ means $x\le
Cy,$ $x\gtrsim y$ means $x\ge cy$ and $x\eqsim y$ means $cx\leq y\le
Cx$ where $c$ and $C$ are generic positive constants independent of
the variables appearing in the inequalities and any other parameters
related to mesh, space and coefficients.

\section{Preliminaries}
\label{sec:pre}
In this section, we introduce some notation, set up our problem, and review briefly some facts about the preconditioned conjugate gradient algorithm. 

\subsection{Notation and Problem}
Given a set of positive constants $\{a_i\}_{i=1}^M,$ we define the following weighted inner products on the space $H^1(\Omega)$
$$(u,v)_{0,a}=\sum_{i=1}^M a_i (u, v)_{L^2(\Omega_i)}, \;\;\mbox{and}\;\;
(u,v)_{1,a}=\sum_{i=1}^M a_i(\nabla u,\nabla v)_{L^2(\Omega_i)}$$ 
with the induced weighted $L^{2}$ norm  $\|\cdot\|_{0,a},$ and the weighted $H^{1}$-seminorm $|\cdot|_{1,a},$ respectively. We denote by $$\|u\|_{1,a}=\left(\|u\|_{0,a}^2 + |u|_{1,a}^2\right)^{\hf},$$
and the related inner product and the induced energy norm by
$$
(u,v)_A=A(u,v) :=(u,v)_{1,a},
\;\quad \|u\|_A=\sqrt{A(u,u)}.
$$

To impose the Dirichlet boundary condition in \eqref{eq:model}, we define 
$$
H^1_{g_D,\Gamma _D} = \{v\in H^1(\Omega): v|_{\Gamma _D} = g_D \hbox{ in the trace sense}\},
$$
and $H^1_{D} := H^1_{0,\Gamma _D}$.  Given a shape regular triangulation ${\T}_h$, which could be highly graded, we define $\V_{h}$ as the standard piecewise linear and global continuous finite element space on $\T_h$. 
Given $f\in H^{-1}(\Omega)$ and $g_N\in H^{1/2}(\Gamma _N)$, the linear finite element approximation of \eqref{eq:model} is the function $u\in \V_h\cap H^1_{g_D,\Gamma _D},$ such
that
\begin{equation}\label{eq:fem1}
A(u,v)=\langle f,v\rangle +\int_{\Gamma_N} g_N v, \quad \hbox{for all } v\in \V_h\cap H^1_{D}.
\end{equation}
Given any $u_0\in \V_h\cap H^1_{g_D,\Gamma _D}$, the problem \eqref{eq:fem1} is equivalent to finding $u\in \V_h\cap H^1_{D}$ such that
\begin{equation}\label{eq:fem2}
A(u,v) = \langle f,v\rangle +\int_{\Gamma_N} g_N v - A(u_0,v),\;\; \forall v\in \V_{h}\cap H_D^1.
\end{equation}
We thus consider the space $\V_{h,D}:=\V_h\cap H^1_{D}$. 
The bilinear form $A(\cdot,\cdot)$ will then introduce a symmetric positive definite (with respect to standard $L^2$-inner product) operator, still denoted by $A$, from $\V_{h,D}$ to $\V_{h,D}$ as
$$
(Au, v) = A(u,v).
$$
Define $b\in \V_{h,D}$ as
$$
(b, v) = \langle f,v\rangle + \int_{\Gamma_N} g_N v - A(u_0,v)\quad \forall v\in \V_{h,D}.
$$
We then get the following operator equation on $\V_{h,D}$
\begin{equation}\label{eq:eq}
   Au = b.
\end{equation}
For simplicity, in the remainder of the paper, we should omit the subscript $D$ in $\V_{h,D}$ without ambiguity. 

We are interested in solving equation \eqref{eq:eq} by
the preconditioned conjugate gradient methods with BPX and multigrid preconditioners. Let us now review briefly some basic results concerning the preconditioned conjugate gradient method. 

\subsection{Preconditioned Conjugate Gradient Method}
Let $B$ be a symmetric positive definite (SPD) operator. Applying it
to both sides of \eqref{eq:eq}, we get an equivalent equation
\begin{equation}\label{eq:BA}
BAu=Bb.
\end{equation}
We apply the conjugate gradient method to solve \eqref{eq:BA} and the resulting method is known as the \emph{preconditioned conjugate gradient (PCG)} method, where $B$ is called a \emph{preconditioner}.

Let  $\kappa(BA)=\lambda_{\max}(BA)/\lambda_{\min}(BA)$ be the (generalized) condition number of the preconditioned system $BA.$ Starting from an arbitrary initial guess $u_0$, we have the following well known convergence rate estimate for the $k$th iteration $u_{k}\; (k\ge 1)$ in PCG (see e.g.~\cite{Saad.Y2003})
$$
\frac{\|u-u_k\|_A}{\|u-u_0\|_A} \leq 2\left(\frac{\sqrt{\kappa(BA)}-1}{\sqrt{\kappa(BA)}+1}\right)^{k}.
$$
So if the condition number $\kappa (BA)$ is uniformly bounded, then PCG algorithm converges uniformly. Here the uniformity means the independence of the size of the matrix $A$. Later on, when $A$ is related to equation \eqref{eq:model}, we shall also discuss the uniformity of convergence with respect to the jump of diffusion coefficients.

If there are some isolated small or large eigenvalues, we can sharpen the above convergence rate estimate as stated in the following theorem. 
\begin{theorem}\cite{Axelsson.O2003}\label{th:pcg}
Suppose that $\sigma(BA)=\sigma_0(BA)\cup \sigma_1(BA)$ such that
there are $m$ elements in $\sigma_0(BA)$ and $\alpha \leq \lambda \leq \beta$ for
each $\lambda\in \sigma_1(BA).$ Then
\begin{equation}\label{eq:pcg}
\frac{\|u-u_k\|_A}{\|u-u_0\|_A}\leq 2K
\left(\frac{\sqrt{\beta/\alpha}-1}{\sqrt{\beta/\alpha}+1}\right)^{k-m},
\end{equation}
where
$$
K=\max_{\lambda\in\sigma_1(BA)}\prod_{\mu\in\sigma_0(BA)}
\left|1-{\lambda\over\mu}\right|.
$$
\end{theorem}
If there are only $m$ small eigenvalues in $\sigma_0(BA)$, say
$$0<\lambda_1\le\lambda_2\dots\le\lambda_i\ll
\lambda_{m+1}\le\dots\leq \lambda_n,$$ then
\begin{equation}\label{eq:m}
  K=\prod_{i=1}^m
\left|1-{\lambda_{n}\over\lambda_i}\right|\le
\left({\lambda_{n}\over\lambda_1}-1\right)^m=\left(\kappa(BA)-1\right)^m.
\end{equation}
Therefore the convergence rate of PCG algorithm will be dominated by
the factor $(\sqrt{\beta/\alpha}-1)/(\sqrt{\beta/\alpha}+1),$ i.e. by $\beta/\alpha$ where
$\beta=\lambda_{n}(BA)$ and $\alpha=\lambda_{m+1}(BA).$ We define the ``effective condition number" as follows.
\begin{definition}Let $\V$ be an $n$-dimensional Hilbert space and $T: \V\to \V$ be a symmetric and positive definite operator. For any integer $m\in [1, n-1]$, the $m$th \emph{effective condition number} of $T$ is defined by
$$
\kappa_{m}(T)=\frac{\lambda_{\max}(T)}{\lambda_{m+1}(T)}
$$
where $\lambda_{m+1}(T)$ is the $(m+1)$-th minimal eigenvalue of $T.$
\end{definition}
As a corollary of Theorem \ref{th:pcg}, we have
\begin{equation}\label{eq:conv}
\frac{\|u-u_k\|_A}{\|u-u_0\|_A}\leq 2(\kappa(BA)-1)^m\left(\frac{\sqrt{\kappa_{m}(BA)}-1}{\sqrt{\kappa_{m}(BA)}+1}\right)^{k-m}.
\end{equation}
From \eqref{eq:conv}, given a tolerance $\varepsilon,$ the number of
iterations of the PCG method to reduce the relative error below the tolerance $\varepsilon$ is (cf.~\cite{Axelsson.O1994,Axelsson.O2003})
$$
m+\left\lceil\left(\log\left(\frac{2}{\varepsilon}\right)+ m|\log (\kappa(BA)-1)|\right)/c_0\right\rceil,
$$
where $c_0=\log\left((\sqrt{\kappa_{m}(BA)} +1)/(\sqrt{\kappa_{m}(BA)} -1)\right).$ Therefore if there exists an $m\geq 1$ such that the $m$th effective condition number is bounded uniformly, then the PCG algorithm will still converge almost uniformly, even though the standard condition number $\kappa (BA)$ might be large.

To estimate the effective condition number, in particular $\lambda_{m+1}(A)$, we use a fundamental tool known as the \emph{Courant
``minimax" principle} (see e.g.~\cite{Golub.G;Van-Loan.C1996}).
\begin{theorem}\label{th:minimax}
  Let $\V$ be an $n$-dimensional Hilbert space with inner product $(\cdot,\cdot)_{\V}$ and $T: \V\to\V$ a symmetric positive operator on $\V.$ Suppose
  $\lambda_1\le\lambda_2\le\cdots\le\lambda_n$ are the eigenvalues of $T,$ then
  $$
  \lambda_{m+1}(T)=\max_{\dim (S)=m}\min_{0\neq v\in
  S^{\perp}}\frac{(Tv,v)_{\V}}{(v,v)_{\V}}
  $$ 
  for $i=1,2,\cdots, n-1.$
  Especially, for any subspace $\V_0\subset \V$ with ${\rm
  dim}(\V_0)=n-m$
  \begin{equation}\label{eq:mm}
    \lambda_{m+1}(T)\ge \min_{0\neq v\in \V_0}\frac{(Tv,v)_{\V}}{(v,v)_{\V}}.
  \end{equation}
\end{theorem}
If both $A$ and $B$ are SPD operators, then $BA$ is SPD in the inner product induced by $B^{-1}$ and $A$. Below, we shall apply Theorem \ref{th:minimax} to $T = BA$ and $(u,v)_{\V}:=(B^{-1}u,v)_{L^2}$. Therefore if we have an inequality of the type $(Av,v)\geq c(B^{-1}v,v)$ for all $v$ in a suitable subspace $\V_0$ with ${\rm dim}(\V_0)=n-m$, we can get a lower bound of $\lambda _{m+1}(BA)$.

\section{Local Quasi-interpolation}
\label{sec:interpolation}
The theoretical justification of the robustness of multilevel preconditioners relies on establishing approximation and stability properties of certain interpolation operators. There are two difficulties: one is the locality and stability and another is the robustness with respect to the coefficient. 

The weighted $L^2$-projection $Q_h^a: L^2(\Omega)\to \V_h$ defined by
$
(Q_h^a u, v_h)_{0,a}=(u, v_h)_{0,a} \;\; \forall v_h\in \V_h
$ 
was used in~\cite{Xu.J;Zhu.Y2008,Zhu.Y2008} for the case of uniform refinement.
For the analysis of local multilevel preconditioners, the interpolation operator should preserve certain local structure. Therefore, the weighted $L^{2}$-projection, which is a global operator, is not appropriate. On the other hand, the standard nodal interpolation operator is local but not stable in the energy norm. 
Local quasi-interpolation, such as Scott-Zhang operators~\cite{Scott.R;Zhang.S1990}, are developed to achieve both locality and stability.

However, the stability constant  will in general depend on the jump of diffusion coefficients if we apply the standard quasi-interpolation globally on the whole domain. The value at a vertex is usually defined using a simplex in the patch of this vertex and thus depends on the diffusion coefficient in this simplex. For a vertex shared by several subdomains, this leads to the dependence of the ratio of coefficients. One remedy is to apply the quasi-interpolation on each subdomain and chose a sub-simplex in the quasi-interpolation. Indeed in the original paper \cite{Scott.R;Zhang.S1990}, a $(d-1)$ sub-simplex is used. Such modification is suitable for the interior vertex relative to interfaces for which a common $(d-1)$ sub-simplex on the interface can be used to glue quasi-interpolations in different regions. For vertices on the boundary of the interface, i.e., edges in 3-D and vertices in 2-D, in general there is no common $(d-1)$ sub-simplex but only $(d-2)$ sub-simplex. The trace of $H^1$ functions is not even well defined on $(d-2)$ sub-simplex. For example, the function value of a $H^1$ function at a point can be changed without changing this function. In the discrete level, it can be shown that the trace of a finite element function on a $(d-2)$ sub-simplex can be almost bounded by its Sobolev norm inside. Therefore we can simply set the function values at the vertices of $(d-2)$ sub-simplex to zero to glue quasi-interpolation operators defined in different domain.

Below, we construct a quasi-interpolation operator by gluing Scott-Zhang operators in each subdomains and interfaces, and show that it is stable uniformly with respect to the jump of coefficients and nearly uniform to the mesh size of the triangulation. We stress that this local quasi-interpolation operator is designed for the analysis only, and is not needed in the practical implementation.

\subsection{Notation on Triangulations}
Let us introduce some notation related to the domain and its triangulations. As we mentioned earlier, we assume that the polygonal or polyhedral subdomains $\Omega_i\; (i=1,\cdots, M)$ are open, disjoint to each other, and satisfy $\cup_{i=1}^M
\overline{\Omega}_i=\overline{\Omega}.$ We denote $\Gamma_{ij} = \partial\Omega_{i}\cap \partial \Omega_{j},$ or simply $\Gamma$ if without ambiguity, as the \emph{interface} between two subdomains $\Omega_{i}$ and $\Omega_{j}.$ The subdomains $\Omega _i\; (i=1,\cdots M)$ may possibly have complicated geometry but we assume that they are resolved by an initial \emph{conforming} triangulation $\T_{0}.$ Recall that a triangulation $\T$ is called \emph{conforming} if the intersection of any two elements $\tau$ and $\tau'$ in $\T$ either consists of a common vertex, edge, face (when $d=3$), or empty.

Let $\N, \; \E $ and $\F$ (when $d=3$) denote the set of vertices, edges, and faces of $\T$ respectively. For each vertex $p\in \N,$ we define local patch $\omega_{p} := \cup_{\tau\ni p} \tau$ and, for $\tau\in \T$, $\omega_{\tau} = \cup_{p\in \tau} \omega_{p}$. Similarly, on the $(d-1)$ dimensional interface $\Gamma$, $o_{p}, \; o_{e}$ and  $o_{f}$ denote the intersection of corresponding local patches and the interface. The linear finite element space associated to $\T$ is denoted by $\V(\T),$ or simply $\V.$ More generally, for any subset $\mathcal{S}\subset \T,$ $\V(\mathcal{S})$  denote the finite element subspace restricted to the subset $G$. Similarly, we should denote $\N(G)\subset \N,$ $\E(G)\subset \E$ and $\F(G)\subset \F$ as the set of vertices, edges, and faces in $\overline{G}\subset \overline{\Omega},$ respectively.

For each element $\tau \in \T,$ we define $h_{\tau}= |\tau|^{1/d}$ and $\rho_{\tau}$ for the radius of its inscribed ball. In the whole paper, we assume that the triangulation is \emph{shape regular} in the sense $h_{\tau} \eqsim \rho_{\tau}.$
Let $h$ denote the piecewise constant mesh size function with $h|_{\tau} = h_{\tau},$ and $h_{\min} := \min_{\tau\in \T} h_{\tau}.$ We should also denote $h_{e}$ by the length of an edge $e\in \E$ and $h_{f}$ by the diameter of a face $f\in \F.$ Moreover, we define $h_{p}$ as the diameter of the local patch $\omega_{p}.$ By the shape regularity assumption, for all $e,f,\tau \subset \omega _p$, we have  
$ h_{p} \eqsim h_{e} \eqsim h_{f} \eqsim h_{\tau}.$

\subsection{Technical Lemmas}
For completeness here, we quote some technical lemmas from \cite{Bramble.J;Xu.J1991}, which will be used later for proving the approximation and stability of our  local interpolation operator.

In two dimensions, it is well known that $H^1(\Omega)$ is not embedded into $L^{\infty}(\Omega)$. But for finite element functions, we can control the $L^{\infty}$ norm by its $H^1$-norm with a factor $|\log h_{\min}|^{1/2}$. 
\begin{lemma}[{\cite[Lemma 2.3]{Bramble.J;Xu.J1991}}]
\label{lm:2dtrace-discrete}
	For any subdomain $\Omega_{i} \subset\mathbb{R}^{2}$, let $\V(\Omega_{i})$ be the finite element space based on a shape-regular triangulation $\mcal T$ of $\Omega_{i}$. Then for all $v\in \V(\Omega_{i}),$ it satisfies
	$$\|v\|_{L^{\infty}(\Omega_{i})} \lesssim \left |\log \frac{H_{i}}{h_{\min}}\right |^{1/2} \left(|v|_{H^{1}(\Omega_{i})}  + H_{i}^{-1} \|v\|_{L^{2}(\Omega_{i})} \right),$$
	where $H_{i} = {\rm diam} (\Omega_{i})$ and $h_{\min} := \min_{\tau\in \T} h_{\tau}.$
\end{lemma}

In three dimensions, the trace of an $H^1$-function on an edge is not well defined. But for a finite element function, its $L^2$-norm on an edge can be bounded by its $H^1$-norm with a factor $|\log h_{\min}|^{1/2}$. It is a generalization of Lemma \ref{lm:2dtrace-discrete} to three dimensions in the sense that controlling the norm on a co-dimension 2 boundary manifolds.

\begin{lemma}[{\cite[Lemma 2.4]{Bramble.J;Xu.J1991}}]
\label{lm:3dtrace}
	Given a polyhedral subdomain $\Omega_{i} \subset \mathbb{R}^{3}$, let $E \subset \mathbb{R}$ be any edge of  $\Omega_{i}$  and $\V(\Omega_{i})$ be a finite element space based on a shape-regular triangulation of $\Omega_{i}$. Then for all $v\in \V(\Omega_{i}),$ there holds
	$$\|v\|_{L^{2}(E)} \lesssim \left |\log \frac{H_{i}}{h_{\min}}\right |^{1/2} \left(|v|_{H^{1}(\Omega_{i})}  + H_{i}^{-1} \|v\|_{L^{2}(\Omega_{i})} \right),$$
	where $H_{i} = {\rm diam} (\Omega_{i}).$
\end{lemma}
In the analysis of the local quasi-interpolation in Theorem~\ref{thm:wl2app} below, we should apply Lemma~\ref{lm:2dtrace-discrete} and Lemma~\ref{lm:3dtrace} on each subdomain $\Omega_{i}$, for which the diameter $H_{i} = {\rm diam}(\Omega_{i}) \simeq 1$ is a fixed generic constant. 

\subsection{Stable Local Quasi-Interpolation}
Given a conforming triangulation $\T_{h},$ the Scott-Zhang interpolation operator $\Pi:H^1(\Omega)\to \V(\T_{h})$ can be defined as follows.  For any $p\in \N(\T_{h}),$ we choose a $(d-1)$-simplex $\sigma_{p}\ni p$ in $\T_{h}$. We remark that the choice of $\sigma_{p}$ is not unique (see Section~\ref{sec:defIa} for the particular choice of $\sigma_{p}$ for our purpose). Let $\{\lambda_{\sigma_{p},i}: i=1,\cdots, d\}$ be the barycentric coordinates of $\sigma_{p}.$ One can define the $L^2$-dual
basis $\{\theta_{\sigma_{p},i}: i=1,\cdots, d\}$ of  $\{\lambda_{\sigma_{p},i}: i=1,\cdots, d\},$ namely, $\int_{\sigma_{p}}\theta_{\sigma_{p},i} \lambda_{\sigma_{p},j}=\delta_{ij}.$
We define a quasi-interpolation $\Pi$ as
\begin{equation}
	\label{eqn:sz}
	\Pi v=\sum_{p\in \N(\T_{h})}\left(\int_{\sigma_{p}}\theta_{\sigma_{p}} v\right)\phi_p,
\end{equation}
where $\{\phi_{p}\}_{p\in \N(\T_{h})}$ is the set of nodal basis of $\V(\T_{h}),$ and $\theta_{\sigma_{p}} =\theta_{\sigma_{p}, 1}$.
%
%
The following properties of the operator $\Pi$ can be found in
\cite{Scott.R;Zhang.S1990,Oswald.P1994}.
\begin{lemma}\label{lm:int}
 The interpolation operator $\Pi$ satisfies the following properties:
\begin{enumerate}
\item {Stability: }
 	\begin{align}
	 &\|\Pi v\|_{L^2(\tau)}\lesssim \|v\|_{L^2(\omega_{\tau})}, &\nonumber\\
	 &\|\Pi v\|_{H^{1}(\tau)} \lesssim \|v\|_{H^{1}(\omega_{\tau})} ;&\label{eqn:sz-stab}
	 \end{align}
\item {Locality: } 
 	\begin{align} 
	(\Pi v)|_{\tau} = v |_{\tau}  \quad \mbox{ if }v\in \V(\omega_{\tau});& \label{eqn:sz-linear}
	\end{align}
\item {Approximability: }
 	\begin{align} 
		\|h^{-1} (v-\Pi v)\|_{L^2(\tau)} \lesssim \|v\|_{H^1(\omega_{\tau})}.&\label{eqn:sz-app}
	\end{align}
\end{enumerate}
\end{lemma}

We apply the quasi-interpolation \eqref{eqn:sz} on each subdomain, and denote $\Pi_{i}: L^{2}(\Omega_{i}) \to \V({\Omega_{i}})$ by the Scott-Zhang interpolation restricted to $\Omega_{i}$. To be able to glue them together, we require for a vertex on the interior of the interface, we choose a common $(d-1)$ sub-simplex shared by two sub-domains. By such choice, $\Pi _i$ and $\Pi _j$ will match on the vertex interior relative to the interface.


We now define a local interpolation operator $\I_{h}^{a}$ which has the desirable local approximation and stability properties in the weighted Sobolev norms. Given a $u\in H^{1}(\Omega),$ we define $\I_{h}^{a} u \in \V(\T_{h})$ such that for $p\in \N(\Omega_{i})$
\begin{equation} 
\label{eq:loc-int}
	\I_{h}^{a} u (p) :=\left\{\begin{array}{ll}
      (\Pi_{i}u) (p),  & \text{otherwise},\\
      0,    &  \text{if } p\in \N(\partial \Gamma_{i}).
    \end{array}\right.
\end{equation}

For a vertex $p,$ let $\sigma_{p}$ be the $(d-1)$-simplex chosen to define the nodal value at $p$. Then the interpolant $\mathcal I_h^a$ is uniquely determined by the mapping $p \to \sigma _p$. In \eqref{eq:loc-int}, if $p$ is in the interior of some subdomain $\Omega_{i},$ then $\sigma_{p}\subset \Omega_{i}$ is chosen to be any $(d-1)$-simplex in $\T$ containing $p$; if $p$ is in the interior of the interface $\Gamma,$ then $\sigma_{p}\subset \Gamma$ is chosen to be a $(d-1)$-simplex on the interface containing $p$. The choice of $\sigma _p$ is not unique. However, in order to preserve the local structure of the adaptive grids, $\sigma_p$ should be chosen carefully for each vertex $p$. This will be clear in Section \ref{sec:decomp} when we discuss the  geometry of the bisection grids (see Section~\ref{sec:defIa} for details). Now we are in the position to present the main result in this section:
\begin{theorem}\label{thm:wl2app}
	Let $\Omega\subset \mathbb R^d$ with $d = 2$ or 3 and $\T_h$ be a triangulation of $\Omega$ with mesh size $h$. Then for all $u\in H^1(\Omega),$ we have 
	$$\| h^{-1}( u - \I_h^{a} u)\|_{0,a, \Omega} \lesssim |\log h_{\min}|^{1/2}\left\|u\right\|_{1,a, \Omega}.$$
\end{theorem}
\begin{proof}
Using the discrete Sobolev inequality Lemma \ref{lm:2dtrace-discrete} or \ref{lm:3dtrace} on $\partial \Gamma$ and the local $H^{1}$-stability \eqref{eqn:sz-stab} of $\Pi_{i}$, we have
\begin{eqnarray*}
	\sum_{\Gamma\subset \partial \Omega_i}
      \|\Pi_i u \|_{L^2(\partial \Gamma)}&\lesssim&  |\log
      h_{\min}|^{1/2} \left\|\Pi_i u\right\|_{H^1(\Omega_i)}
      \lesssim |\log
      h_{\min}|^{1/2} \left\|u\right\|_{H^1(\Omega_i)}. 
  \end{eqnarray*}

By the triangle inequality and the approximation property \eqref{eqn:sz-app} of $\Pi_{i}$, we have
\begin{eqnarray*}
&&\|h^{-1}(u  - \I_h^a u)\|_{L^2(\Omega_{i})}\\
 &&\le \| h^{-1}(u - \Pi _i u)\|_{L^2(\Omega_{i})}+ \|h^{-1}(\Pi _i u - \I_{h}^{a} u)\|_{L^2(\Omega_{i})}\\
&&\lesssim \|u\|_{H^{1}(\Omega_{i})} +  \sum_{\Gamma \subset \partial \Omega_i}  \|\Pi_i u \|_{L^2(\partial \Gamma)}\\
&&\lesssim \|u\|_{H^{1}(\Omega_{i})} + |\log h_{\min}|^{\frac{1}{2}} \left\|u\right\|_{H^1(\Omega_i)}.
\end{eqnarray*}

Multiplying by a suitable weight and summing
up over all subdomains on both sides, we get the desired estimate. $\Box$
\end{proof}

In general, we cannot replace $\left\|u\right\|_{1,a}$ by the
energy norm $\left|u\right|_{1,a}$ in the above lemma; see~\cite{Xu.J1991} for a counter example. To be able to use $\left|u\right|_{1,a}$ in the estimate, we introduce a subspace $\widetilde{H}^1_{D}(\Omega)$ of $H_{D}^1(\Omega)$ as follows:
$$\widetilde{H}_{D}^1(\Omega)=\left\{u\in H_{D}^1(\Omega): \int_{\Omega_i} u \, \dx=0 \quad \hbox{ for all } i\in I\right\},$$
where $I$ is the set of indices of all \emph{floating} subdomains:
$$I=\{i:\; \mbox{meas}(\partial \Omega_i\cap \Gamma_D)=0\}.$$
Let $m_{0}:=\#I$ be the cardinality of $I.$ We emphasize that $m_{0}$  is a constant, depending only on the distribution of the coefficients, and $m_0\leq M.$
In this subspace $\widetilde{H}_{D}^1(\Omega),$ the interpolation $\I_{h}^{a}$ has the following properties.
\begin{theorem}\label{thm:wl2}
  For any $v\in \widetilde{H}^1_D(\Omega)$, we have the approximation property of $\I_h^a$
  \begin{equation}\label{eq:wl2app}
    \left\|h^{-1}(v-\I_h^a v)\right\|_{0,{a}}\lesssim
    \left|\log h_{\min}\right|^{\frac{1}{2}} \left|v\right|_{1,{a}},
  \end{equation}
and the stability of $\I_h^a$ in the energy norm
  \begin{equation}\label{eq:wl2stab}
    \left|\I_h^a v\right|_{1,a}\lesssim
    \left|\log h_{\min}\right|^{\frac{1}{2}} \left|v\right|_{1,{a}}.
  \end{equation}
\end{theorem}
\begin{proof}
   For $v\in \widetilde{H}^1_D(\Omega)$, it satisfies the Poincar\'e-Friedrichs inequality on each subdomain $\Omega_i$. Therefore we get $\left\|v\right\|_{0,a}\lesssim  \left|v\right|_{1,a}.$ The inequality
  \eqref{eq:wl2app} then follows from Lemma \ref{thm:wl2app}.

  To prove inequality \eqref{eq:wl2stab}, we use the inequality \eqref{eq:wl2app} and the local $L^2$ projection $Q_\tau: L^2(\tau) \to \P_0(\tau)$ defined by $Q_\tau u |_{\tau}= |\tau|^{-1}\int _{\tau} u \dx.$ Then on each element $\tau\in {\T}_h,$ we have
  \begin{eqnarray*}
    \left|\I_h^a v\right|_{H^1(\tau)}^2&\lesssim& \left|\I_h^a v-Q_\tau v\right|_{H^1(\tau)}^2\lesssim h_{\tau}^{-2} \left\|\I_h^a v-Q_\tau v\right\|_{L^2(\tau)}^2\\
    &\lesssim& h_{\tau}^{-2} \left (\left\|v-\I_h^a v\right\|_{L^2(\tau)}^2+\left\|v-Q_\tau v\right\|_{L^2(\tau)}^2 \right )\\
    &\lesssim& h_{\tau}^{-2} \left\|v-\I_h^a v\right\|_{L^2(\tau)}^2+\left|v\right|_{H^1(\tau)}^2
  \end{eqnarray*}
  where in the last inequality, we used the approximation properties of $Q_\tau$. Multiplying by a suitable weight and summing
  up over all $\tau\in {\T}$ on both sides, we get
   \begin{eqnarray*}
    \left|\I_h^a v\right|_{1,a}^2
    \lesssim  \left\| h^{-1}(v-\I_h^a
    v)\right\|_{0,a}^2+\left|v\right|_{1,a}^2
    \lesssim \left|\log h_{\min}\right|\left|v\right|_{1,a}^2
  \end{eqnarray*}
  where in the last step, we used inequality
  \eqref{eq:wl2app}. $\Box$
\end{proof}

\begin{remark}
When the coefficients satisfy the quasi-monotone assumption, the factor $|\log h_{\min}|$ can be removed by arguments on a modified local patch; see~\cite{Dryja.M;Sarkis.M;Widlund.O1996,Petzoldt.M2002}.
$\Box$\end{remark}

\section{Bisection Grids and Space Decomposition}
\label{sec:decomp}
In this section, we give a short overview of the framework in
the multilevel space decomposition on bisection grids in the recent
work~\cite{Chen.L;Nochetto.R;Xu.J2007,Xu.J;Chen.L;Nochetto.R2009}. Most of the material in
this section can be found there.

\subsection{Bisection Methods} We recall briefly the bisection algorithm for the mesh refinements. Detailed discussions can be found in
\cite{Binev.P;Dahmen.W;Devore.R2004,Chen.L2006a,Mitchell.W1992} and
the references cited therein.

 Given a conforming triangulation
$\T$ of $\Omega,$ for each element $\tau\in \T,$ we assign an edge of $\tau$ to be the \emph{refinement
edge} of $\tau$, denoted by $e(\tau)$ or simply $e$ without ambiguity. This procedure is called \emph{labeling}. Given a set of elements marked for refinement, the refinement procedure consists
two steps:
\begin{enumerate}
  \item[(1)] bisect the marked element into two elements by
  connecting the middle point of the refinement edge to the vertices not contained in the refinement edge;
  \item[(2)] assign refinement edges for two new elements.
\end{enumerate}

Given a labeled initial grid $\T_0$ of
$\Omega$ and a bisection method, we define
\begin{align*}
  \mathbb F(\T_0) &= \{\T : \T \hbox{ is refined from } \T_0 \hbox{ by bisection method }\},\\
\mathbb T(\T_0) & = \{\T\in \mathbb F(\T_0): \T \text{ is conforming} \}.
\end{align*}
Namely $\mathbb F(\T_0)$ contains all triangulations obtained
from $\T_0$ using the chosen bisection method. But a triangulation
$\T\in \mathbb F(\T_0)$ could be non-conforming and thus we
define $\mathbb T(\T_0)$ as a subset of $\mathbb F(\T_0)$
containing only conforming triangulations. 

Given any triangulation $\T$, we define $\overline{\T}_{0} = \T$, and the $k$th uniform refinement $\overline{\T}_{k}\;\; (k\ge 1)$ being the triangulation obtained by bisecting all element in $\overline{\T}_{k-1}$ only once. Note that for a conforming initial triangulation $\T_0$ with arbitrary labeling, $\overline{\T}_k\in \mathbb F(\T_0)$ but not necessarily in the set $\mathbb T(\T_0)$ in general. Throughout this paper, we shall consider bisection methods which satisfy the
following two assumptions:

\medskip

\noindent\textbf{(B1) Shape Regularity:}
 $\mathbb F(\T_0)$ is shape regular.

\smallskip

\noindent \textbf{(B2) Conformity of Uniform Refinement:}
$\overline{\mathcal T}_k(\T_0)\in \mathbb T(\T_0)$ for all $k\geq 0$.

\medskip

In two dimensions, newest vertex
bisection with compatible initial labeling~\cite{Mitchell.W1989} satisfies (B1) and (B2). In three and higher dimensions, the bisection method by Kossaczk\'y~\cite{Kossaczky.I1994} and Stevenson~\cite{Stevenson.R2008} will satisfy (B1) and (B2). 
We note that to satisfy assumption (B2), the initial triangulation is modified by further refinement of each element, which deteriorates the shape regularity. 
Although (B2) imposes a severe restriction on the initial labeling, it is crucial to control the number of elements added in the completion which is indispensable to establish the optimal complexity of adaptive finite element methods~\cite{Nochetto.R;Siebert.K;Veeser.A2009}. 

\subsection{Compatible Bisections}
For a vertex $p\in \N(\T)$ or an edge $e\in \E(\T)$, we define the {\em first
ring} of $p$ or $e$ to be
$$
 \R _p  = \{ \tau  \in \T \,|\, p\in  \tau \},
 \quad \R_e = \{ \tau  \in \T \,|\, e\subset  \tau \},
$$
and the local patch of $p$ or $e$ as $\omega _p = \cup _{ \tau  \in \R_p}  \tau, $ and $\omega _e = \cup _{\tau  \in \R_e}  \tau .$
Note that $\omega _p$ and $\omega _e$ are subsets of $\Omega$, while
$\R_p$ and $\R_e$ are subsets of $\T$ which can be
thought of as triangulations of $\omega _p$ and $\omega _e$,
respectively. The cardinality of a set $S$ will be denoted by $\# S$.

Given a labeled triangulation $\T$, an edge $e\in
\E(\T)$ is called a \textit{compatible edge} if $e$ is the
refinement edge of $\tau$ for all $\tau\in \R_e$. For a
compatible edge, the ring $\R_e$ is called a \textit{compatible
  ring}, and the patch $\omega _e$ is called a \textit{compatible patch}. Let
$p$ be the midpoint of $e$ and $\R_p$ be the ring of $p$ in the refined triangulation. A \textit{compatible
  bisection} is a mapping $b_e: \R_e \to \R_p.$
We then define the addition 
$$
\T + b_e := \T \backslash \R_e \cup \R_p.
$$
For a compatible bisection sequence $\B:=(b_{1}, \cdots, b_{k})$, the addition $\T + \B$ is defined as 
$$\T + \B = ((\T + b_{1}) + b_2) + \cdots + b_k,$$
whenever the addition is well defined. Note that if $\T$ is conforming, then $\T + b_e$ is conforming
for a compatible bisection $b_e$, whence compatible bisections
preserve the conformity of triangulations.

We now present a decomposition of meshes in  $\mathbb
T(\T_0)$ using compatible bisections, which will be instrumental later. We only give a pictorial demonstration in Fig.~\ref{Fig:decomp} to illustrate the decomposition. For the proof, we refer to~\cite{Xu.J;Chen.L;Nochetto.R2009}. 

\begin{theorem}[Decomposition of Bisection Grids]\label{th:decomposition} 
  Let $\T_0$ be a conforming triangulation. 
  Suppose the bisection method satisfies
  assumptions (B2), i.e., for all $k\geq 0$ all uniform refinements
  $\overline{\T}_k$ of $\T_0$ are conforming. Then for any
  $\T \in \mathbb T(\T_0)$, there exists a compatible
  bisection sequence $\B=(b_1, \cdots, b_N)$ with
  $N=\#\mcal N(\T)-\#\mcal N(\T_0)$ such that
\begin{equation}\label{eq:Tdec} 
\T = \T_0 + \B.
\end{equation}
\end{theorem}
\begin{figure}[h!!]
\begin{center}
\includegraphics*[width=0.88\linewidth]{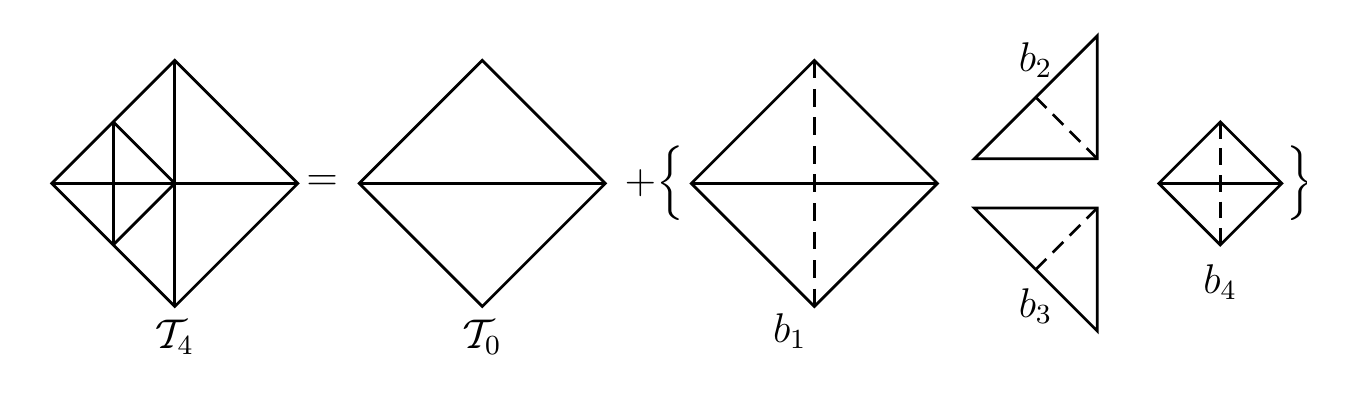}
\caption{A decomposition of a bisection grid.}
\label{Fig:decomp}
\end{center}
\end{figure}

We point out that in practice it is not necessary to store
$\B$ explicitly during the refinement procedure. Instead we can apply coarsening algorithms to find the decomposition. We refer to~\cite{Chen.L;Zhang.C2007} (see also~\cite{Chen.L2008}) for a vertex-oriented coarsening algorithm and the application to multilevel preconditioners and multigrid methods.

For a compatible bisection $b_i\in \B$, we use the same subscript $i$ to denote related quantities such as:

\smallskip
\parbox{5.5cm}{
\begin{itemize}
\item $e_i$: the refinement edge;
\item $p_i$: the midpoint of $e_i$;
\item $\widetilde{\omega} _i= \omega _{p_i}\cup \omega _{p_{l_i}}\cup \omega _{p_{r_i}}$;
\item $\T_i=\T_0+ (b_1, \cdots, b_i)$;
\end{itemize}}
\parbox{6cm}{
  \begin{itemize}
\item $\omega _i$: the patch of $p_i$ i.e. $\omega _{p_i}$;
\item $p_{l_i}, p_{r_i}$: two end points of $e_i$;
\item $h_i$: the diameter of $\omega _i$;
\item $\R_i$: the first ring of $p_i$ in $\T_i$.
  \end{itemize}
}
\subsection{Generation of Compatible Bisections}
The generation of each element in the initial grid $\T_0$ is defined to be $0$, and the generation of a child is 1 plus that of the father. The generation of an element $\tau \in \T\in \mathbb F(\T_0)$ is denoted by
$g_{\tau}$ and coincides with the number of bisections needed to
create $\tau$ from $\T_0$. For
any vertex $p\in \mathcal {N}(\T_0)$, the generation of $p$ is
defined as the minimal integer $k$ such that $p\in\mathcal N(\overline{\T}_k)$ and is denoted by $g_p$. 
In \cite{Xu.J;Chen.L;Nochetto.R2009}, we show that if $b_i\in \B$ is a compatible bisection, then all elements of $\R_i$ have the same generation $g_i$. Therefore we can introduce the concept of generation of compatible bisections. For a compatible bisection $b_i: \R_{e_i}\to \R_{p_i}$, we define $g_i=g(\tau), \tau \in \R_{p_i}$. 

Throughout this paper we always assume $h(\tau)\eqsim 1$ for $\tau \in \T_0$. Then since a bisection of a simplex will reduce the volume by half, we have the following important relation between generation and mesh size
\begin{equation*}
h_i \eqsim \gamma ^{\,g_i}, \; \text{ with }\, \gamma = \Big (\frac{1}{2}\Big )^{1/d}\in (0,1).
\end{equation*}
In particular, we introduce a ``level'' (or generation) constant $L:= \max_{\tau\in \T} g_{\tau}.$ It is obvious that $L\eqsim \lceil |\log h_{\min}|\rceil.$

Different bisections with the same generation have disjoint 
local patches. Namely for two compatible bisections $b_i$ and $b_j$ with $g_j=g_i$, we then have $\omega _i \cap \omega _j = \varnothing.$
A simple but important consequence is that, for all $u\in L^2(\Omega)$ and $k\geq 0$, 
\begin{align}
\label{eq:sumtilde} \sum _{g_i = k} \| u \|^2_{0,a,\widetilde{\omega}_i} 
\lesssim \|u\|_{0,a, \Omega}^2.
\end{align}

\subsection{A Local Quasi-Interpolation} 
\label{sec:defIa}

We define a sequence of quasi-interpolation operators recursively. Let $\I^{a}_{0}: \V(\T_{N}) \to \V_{0}$ be an arbitrary interpolation operator defined by \eqref{eq:loc-int}. Assume $\I^{a}_{i-1}: \V(\T_{N})\to \V(\T_{i-1})$ is defined. Let $b_{i}$ be a compatible bisection, which introduces a new vertex $p_{i}$ from $\T_{i-1}$ to $\T_{i} = \T_{i-1} + b_{i}.$ We construct $\I^{a}_{i}: \V(\T_{N}) \to \V(\T_{i})$ as follows.
If the new vertex $p_{i} \in \Gamma_{D}$, we simply define $(\I^{a}_{i} v)(p_{i}) = 0$ to reflect the vanishing boundary condition of $v.$ Otherwise, if $p_{i}\notin \Gamma_D$ we define the nodal value at $p_{i}$ through \eqref{eqn:sz} with the choice of $\sigma_{p_{i}}$ as follows:
\begin{enumerate}
	\item if $p_{i}$ is in the interior of some subdomain $\Omega_{i},$ we choose a $(d-1)$-simplex $\sigma_{p_{i}}$ containing $p_{i}$;
	\item if $p_{i}$ is in the interior of some interface $\Gamma,$ we choose a $(d-1)$-simplex $\sigma_{p_{i}}\subset \Gamma$ containing $p_{i}$;
	\item otherwise, we simply let $\sigma_{p_{i}} = \emptyset$ and define  $(\I^{a}_{i}v)(p_{i}) = 0.$
\end{enumerate} 
For other vertices $p\in \N(\T_{i-1}),$ let $\sigma_p\in \T_{i-1}$ be the simplex
used to define $(\I^{a}_{i-1}v)(p),$ we update $(\I^{a}_{i}v)(p)$
according to the following two cases:
\begin{enumerate}
\item if $\sigma _p\subset \overline{\omega _p(\T_i)}$ we keep the nodal
  value, i.e., $(\I^{a}_{i}v)(p)=(\I^{a}_{i-1}v)(p)$;
\item otherwise we update $\sigma_{p}$ as $\sigma_p \leftarrow \overline{\omega _p(\T_i)}\cap \sigma_{p}$ to define $(\I^{a}_{i}v)(p).$
\end{enumerate}
In either case, we ensure that the simplex $\sigma_p \subset \overline{\omega _p(\T_i)}.$ In this way, we obtain a sequence of quasi-interpolation
operators 
$$
\I^{a}_{i}: \V(\T_N)\to \V(\T_i),\quad
i=0\cdots N.
$$
Note that in general $\I^{a}_N v\neq v$ since the simplex used to define nodal values of $\I^{a}_N v$ may not be in the finest mesh $\T_N$ but in $\T_{N-1}.$ 
Figure \ref{fig:updatepatch} illustrates the choice of $\sigma_{p}$ in different cases in 2D.
\begin{figure}[hdpt]
\subfigure[Simplex to define $(\I^{a}_i u)(p_i)$]{
\label{fig:update1}
\begin{minipage}[t]{0.48\linewidth}
\centering
\includegraphics*[width=5.5cm]{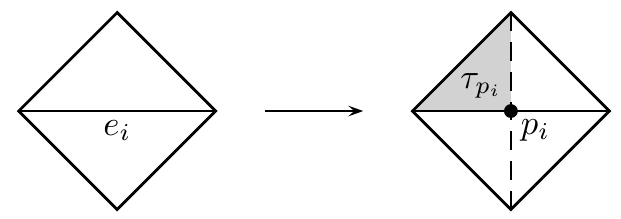}
\end{minipage}}
\subfigure[Simplex to define $(\I^{a}_i u)(p_{l_i})$]{
\label{fig:update2}
\begin{minipage}[t]{0.48\linewidth}
\centering
\includegraphics*[width=5.5cm]{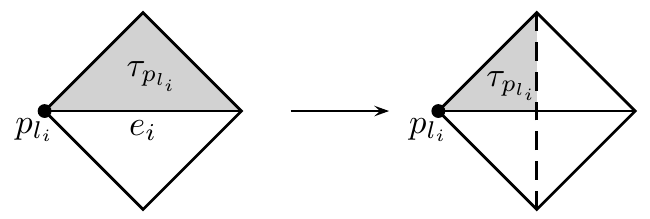}
\end{minipage}}
\subfigure[Simplex to define $(\I^{a}_i u)(p_{r_i})$]{
\label{fig:update3}
\begin{minipage}[t]{0.48\linewidth}
\centering
\includegraphics*[width=5.5cm]{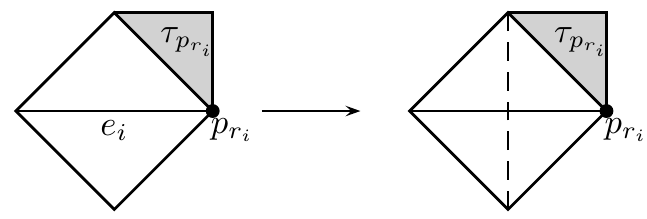}
\end{minipage}}
\subfigure[Simplex to define $(\I^{a}_i u)(p)$]{
\label{fig:update4}
\begin{minipage}[t]{0.48\linewidth}
\centering
\includegraphics*[width=5.5cm]{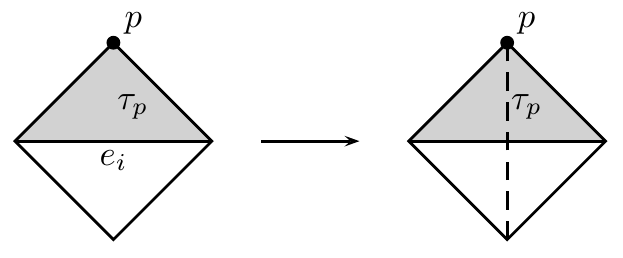}
\end{minipage}}
\caption{Update of nodal values $\I^{a}_iu$ to yield $\I^{a}_{i-1}u$:
the element $\tau$ chosen to perform the averaging that gives
$(\I^{a}_iu)(p)$ must belong to $\omega_p(\T_i)$. This implies
$(\I^{a}_i-\I^{a}_{i-1})u(p)\ne0$ possibly for $p=p_i,p_{l_i},p_{r_i}$
and $=0$ otherwise.}
\label{fig:updatepatch}
\end{figure}

\subsection{Stable Space Decomposition} 
Let $\phi_{i,p}\in \V(\T_i)$ denote the nodal basis at
node $p\in \N(\T_i).$ Motivated by the stable three-point wavelet construction by Stevenson~\cite{Stevenson.R1998}, we define the subspaces
$\V_0=\V(\T_0),$ and
$$\V_i={\rm span}\{\phi_{i,p_{i}}, \phi_{i, p_{l_{i}}}, \phi_{i,p_{r_{i}}}\}.$$ 
Let $\{\phi_{p} : p\in \Lambda\}$ be a basis of $\V(\T_{N}),$ where $\Lambda$ is the index set of the basis functions, and let $\mathcal V_p$ be the 1-dimensional subspace spanned by the nodal bases  associated to $p$ in the finest grid. We choose the following space decomposition:
\begin{equation}\label{eq:V}
\V :=\sum_{p\in \Lambda} \V_{p} + \sum_{i=0}^{N} \V_{i}.
\end{equation}

Recall that $b_i$ only changes the local patches of two end points of the refinement edge $e_i$ going from $\T_{i-1}$ to $\T_{i}.$ By construction $(\I^{a}_{i}-\I^{a}_{i-1})v(p)=0$ for $p\in \N(\T_i), p\neq p_i, p_{l_i}$ or $p_{r_i}$, which implies $v_i:=(\I^{a}_{i}-\I^{a}_{i-1})v\in \V_i.$ Although $\I^{a}_N v\neq v$ in general, the difference $v-\I^{a}_{N}v$ is of high frequency in the finest mesh. Let us write $v-\I^{a}_{N}v =\sum _{p\in \Lambda} v_p$ as the basis decomposition.  We then obtain a decomposition 
\begin{equation}
  \label{eq:proofdec}
  v= \sum _{p\in \Lambda}v_p + \sum _{i=0}^Nv_i, \quad v_i \in \mathcal V_i,
\end{equation}
where  for convenience we define $\I^{a}_{-1}:=0.$ Moreover, we introduce a subspace $\widetilde{\V}:= \V \cap \widetilde{H}^{1}_{D}(\Omega).$
 Then we have the following stable decomposition. 
 \begin{theorem}[Stable Decomposition]\label{th:stda}
Given a triangulation $\T_N = \T_0 + \B$ in $\mbb T(\T_0)$, let $L=\max_{\tau\in \T_N}g(\tau)$. 
  \begin{enumerate}
    \item For any $v\in \V,$ there exist $v_{p} \in \V_{p}\;(p\in \Lambda)$ and $v_i\in \V_i \; (i=1, \cdots,
    N)$ such that $v=\sum_{p\in \Lambda} v_{p} + \sum_{i=0}^N v_i$ and
    \begin{equation}\label{eq:stdav}
      \sum_{p\in \Lambda} h_{p}^{-2} \|v_{p}\|_{0,a}^{2} + \|v_0\|_{1,a}^2+\sum_{i=1}^N h_i^{-2}\|v_i\|_{0,a}^2
      \lesssim c_d(L) |v|_{1,a}^2,
    \end{equation}
    where $ c_d(L)= \left\{\begin{array}{ll}
    L^2, &d=2\\
    2^L, &d=3
  \end{array}\right..$
    \item For any $v\in \widetilde{\V},$ there exist $v_{p} \in \V_{p}\;(p\in \Lambda)$ and $v_i\in \V_i \; (i=1, \cdots,
    N)$ such that $v=\sum_{p\in \Lambda} v_{p} + \sum_{i=0}^N v_i$ and
    \begin{equation}\label{eq:stdatv}
      \sum_{p\in \Lambda} h_{p}^{-2} \|v_{p}\|_{0,a}^{2} + \|v_0\|_{1,a}^2+\sum_{i=1}^N h_i^{-2}\|v_i\|_{0,a}^2
      \lesssim L^2 |v|_{1,a}^2
    \end{equation}
  \end{enumerate}
\end{theorem}
\begin{proof}
The result of (i) is standard. We may use the standard nodal interpolation operator to define a decomposition using the hierarchical basis (cf.~\cite{Xu.J1997b}). 

Now we prove (ii). Given a $v\in \widetilde{\V},$ we define $v_{0} := \I_{0}^{a}v$ and $v_{i} : = (\I_{i}^{a}- \I_{i-1}^{a})v.$ 
For $v-\I_{N}^{a} v = \sum_{p\in\Lambda}v_p,$ by the approximability of the quasi-interpolation, cf. \eqref{eq:wl2app}, we have  
\begin{equation}\label{eq:IT}
	\sum _{p\in \Lambda} h_p^{-2}\|v_p\|^2_{0,a} \lesssim \|h^{-1}(v-\I_N^{a}v)\|_{0,a}^2\lesssim L |v|_{1,a}^2.
\end{equation}
On the other hand, by Theorem \ref{thm:wl2} we obtain
\begin{eqnarray*}
    &&\left \|\I_0^a v\right \|_{1,a}^2+\sum_{i=1}^N
    h_i^{-2}\|(\I_i^a-\I_{i-1}^a)v\|_{0,a,\omega _i}^2 \\
    &&=\left \|\I_0^a v\right \|_{1,a}^2+\sum_{l=1}^L \sum_{g_{i} = l}
    h_l^{-2}\|(\I_i^a-\I_{i-1}^a)v\|_{0,a,\omega _i}^2 \\
    &&\lesssim
    \left(\sum_{i=1}^L |\log h_{\min}|\right) \|v\|_{1,a}^2\lesssim L^2 |v|_{1,a}^2.
 \end{eqnarray*}
Then \eqref{eq:stdatv} follows by adding the above inequality to inequality \eqref{eq:IT}. $\Box$
\end{proof}
\begin{remark}
The estimate \eqref{eq:stdav} is not uniform for $d\geq 2$. For $d=2$, $L\approx |\log h_{\min}|$ and the growth of $c_2(L)$ is acceptable. But for $d=3$, the constant $c_3(L)=2^L$ grows exponentially. This is the main reason that the hierarchical basis multilevel method deteriorates rapidly in 3D (cf. \cite{Yserentant.H1993,Bank.R1996}). For discontinuous coefficients problems, it seems unlikely to find a better decomposition with a better constants; see the counterexamples in~\cite{Bramble.J;Xu.J1991,Oswald.P1999c}.

If the coefficients satisfy certain monotonicity, e.g. quasi-monotonicity (cf.~\cite{Dryja.M;Sarkis.M;Widlund.O1996,Petzoldt.M2002}) in the local patches, one can show that the interpolation operator defined above is stable in the energy norm without deterioration. 
$\Box$\end{remark}

\begin{remark}
	\label{rk:stable}
	 With a close look at the proof of \eqref{eq:stdatv}, we may regroup the $v_{i}=(\I_i^a-\I_{i-1}^a)v$ into groups $\cup _{l=1}^{L'} G(l)= \{1, 2, \cdots, N\}$ such that for any $i,j\in G(l), \omega_{j} \cap \omega_{i} =\varnothing$ and therefore
	\begin{eqnarray*}
		\sum _{i=1}^N h_i^{-2} \|v_i\|_{0,a,\omega _i}^2 &=& \sum_{l=1}^{L'} \sum_{j\in G(l)} h_{j}^{-2}\|v_{j}\|_{0,a,\omega _i}^{2} 
	\leq L'|\log h_{\min}| |v|_{1,a}^2.
	\end{eqnarray*}
The constant $L'$ could be much smaller than $L$; see Section~\ref{sec:num} for numerical examples.
$\Box$\end{remark}

\subsection{Strengthened Cauchy-Schwarz inequality}
An important tool in analysis of the multiplicative preconditioner is the
following strengthened Cauchy-Schwarz inequality. A proof can be found in~\cite{Chen.L;Nochetto.R;Xu.J2007,Xu.J;Chen.L;Nochetto.R2009}.

\begin{lemma}[Strengthened Cauchy-Schwarz Inequality]\label{lm:scs}
    For any $u_i, v_i\in \V_i,\;\; i=0,1, \cdots, N,$ we have
  \begin{equation}\label{eq:scs}
    \left|\sum_{i=0}^N\sum_{j=i+1}^N A(u_i, v_j)\right|\lesssim \left(\sum_{i=0}^N
    |u_i|_{1,a}^2\right)^{\frac{1}{2}}\left(\sum_{i=0}^N
    h_i^{-2}\|v_i\|_{0,a}^2\right)^{\frac{1}{2}}.
  \end{equation}
\end{lemma}

As a corollary of \eqref{eq:scs} and the inverse inequality, we have
\begin{equation}\label{eq:sumui}
\Big \|\sum _{i=0}^N u_i \Big \|_{1,a}^2 \lesssim \sum _{i=0}^Nh_i^{-2}\|u_i\|_{0,a}^2.
\end{equation}

\section{Multilevel Preconditioners} 
\label{sec:precond}
In this section, we shall analysis the eigenvalue distribution of the BPX preconditioner and the multigrid $V$-cycle preconditioner on bisection grids, and prove the effective conditioner number is uniformly bounded.

\subsection{BPX (Additive) Preconditioner}
To simplify the notation, we include $\mathcal V_{N+1} = \mathcal V$ and rewrite our space decomposition as $\mathcal V = \sum _{i=0}^{N+1} \V_{i}.$ Based on this space decomposition, we choose SPD smoothers $R_i: \mathcal V_i\to \mathcal V_i$ satisfying
\begin{equation}\label{eq:smbpx}
(R_i^{-1} u_i,u_i)_{0,a}\eqsim h^{-2}_i
  (u_i,u_i)_{0,a},
  \;\; \forall u_i\in \V_i .
\end{equation}

According to~\cite{Xu.J;Zhu.Y2008}, both of the standard Jacobi and
symmetric Gauss-Seidel smoother satisfy the above assumption. On
the coarsest level, i.e. when $i=0$, we choose the exact solver
$R_0=A_0^{-1}.$  Let $Q_i^a: \mathcal V\to \V_i$ be the weighted $L^2$ projection. Then we can define the BPX-type preconditioner
\begin{equation}\label{eq:BPX}
  B = \sum_{i=0}^{N+1} R_i Q_i^a.
\end{equation}

It is well known 
\cite{Widlund.O1992,Xu.J1992,Xu.J;Zikatanov.L2002} that the operator $B$ defined by \eqref{eq:BPX} is SPD, and
\begin{equation}\label{eq:Bidentity}
    (B^{-1}v,v)_{0,a}=\inf_{\sum_{i=0}^{N+1} v_i=v}\sum_{i=0}^{N+1}
    (R_i^{-1}v_i,v_i)_{0,a}.
\end{equation}

We have the following main result for the BPX preconditioner.
\begin{theorem}\label{th:bpx}
Given a triangulation $\T_N = \T_0 + \B$ in $\mbb T(\T_0)$, let $L=\max_{\tau\in \T_N}g(\tau)$. For the BPX preconditioner defined in \eqref{eq:BPX}, we have
  $$\kappa(BA)\leq C_1c_d(L),\hbox{ and } \kappa_{m_0}(BA)\leq  C_0L^2.$$
  Consequently, we have the following convergence estimation of the BPX preconditioned
  conjugate gradient method:
  \begin{eqnarray*}
   \frac{\|u-u_k\|_A}{\|u-u_0\|_A} &\le& 2
\left(C_1c_d(L)-1\right)^{m_0}\left(\frac{C_0L-1}{C_0L+1}\right)^{k-m_0}.
  \end{eqnarray*}
\end{theorem}
\begin{proof}
First of all, let us estimate $\lambda_{\max}(BA).$ For any decomposition $v = \tilde v + \sum _{i=0}^N v_i, \tilde v\in \mathcal V, v_i\in \mathcal V_i$, we have
\begin{align*}
\| v\|_A^2 &\lesssim \|\tilde v\|_A^2 + \Big \|\sum _{i=0}^N v_i \Big \|_A^2 
\leq \|h^{-1}\tilde v\|_{0,a}^2 + \sum _{i=0}^N h_i^{-2}\|v_i\|_{0,a}^2 
\leq \sum_{i=0}^{N+1}(R_i^{-1}v_i,v_i)_{0,a}.
\end{align*}
In the second step, we used the inverse inequality and the inequality \eqref{eq:sumui}. In the third step, we used the assumption \eqref{eq:smbpx} of $R_i$ .  
Taking infimum, we get
  $$\|v\|_A^2\lesssim \inf_{\sum_{i=0}^{N+1} v_i=v}\sum_{i=0}^{N+1}
    (R_i^{-1}v_i,v_i)_{0,a} = (B^{-1}v,v)_{0,a},
  $$ which implies
  that $\lambda_{\max}(BA)\lesssim 1.$  

To estimate $\lambda_{\min}, $ in view of \eqref{eq:Bidentity} we choose the decomposition as in the stable decomposition Theorem~\ref{th:stda}  (see \eqref{eq:stdav}) to conclude that
$$
(B^{-1}v,v)_{0,a}\leq \sum_{i=0}^{N+1}
    (R_i^{-1}v_i,v_i)_{0,a}\lesssim 
    c_d(L)(Av,v)_{0,a},
$$
which implies that $\lambda_{\min}(BA)\gtrsim c_d(L).$    Therefore we have $\kappa(BA)\lesssim c_d(L).$

On the other hand, if we apply \eqref{eq:stdatv} in the subspace $\widetilde{\V}\subset \V,$ we obtain $\lambda_{m_0 +1}(BA)\gtrsim L^2$ by the ``min-max" Theorem
  \ref{th:minimax}. Hence we get an estimate of the effective condition
  number $\kappa_{m_0}(BA)\lesssim L^2.$ The convergence rate
  estimate then follows by Theorem \ref{th:pcg}. This completes the
  proof. $\Box$
\end{proof}

From this convergence result, we can see that the convergence rate
will deteriorate a little bit by $c_d(L)$ as $L$ grows. But since
$m_0$ is a fixed number, when $k$ grows, the convergence rate will
be controlled by the effective condition number, which is bounded uniformly
with respect to the coefficient and logarithmically with respect to the mesh size. Notice
that $L\eqsim |\log h_{\min}|$ and thus the asymptotic convergence rate of the
PCG algorithm is $1-\frac{1}{C|\log h_{\min}|}$ for $h<1.$
\begin{remark}
The estimate $\kappa(BA)\leq C_1c_d(L)$ is sharp in the sense that there exists an example on BPX preconditioner such that $\kappa (BA)\eqsim c_d(L)$ (cf.~\cite{Oswald.P1999c}).
$\Box$\end{remark}

\begin{remark}\label{rk:wst}
  Here we should emphasize that the convergence rate estimate in Theorem \ref{th:bpx} holds for general substructures. In some special circumstance, for example ``edge
  type" or ``exceptional" in the terminology in
 ~\cite{Oswald.P1999c}, or ``quasi-monotone" coefficient in
 ~\cite{Dryja.M;Sarkis.M;Widlund.O1996}, we can sharpen the
  convergence estimate in Theorem \ref{th:bpx} by a modification of Theorem \ref{th:stda},
  see~\cite{Oswald.P1999c}. 
$\Box$\end{remark}

\subsection{Multigrid (Multiplicative) Preconditioner}
We shall use the following symmetric V-cycle multigrid as a preconditioner in the PCG method and prove the efficiency of such a method. Let $A_{i}:= A|_{\V_{i}}.$ Then one step of the standard $V$-cycle multigrid $B:\V \to \V$ is recursively defined as follows:

\smallskip
\framebox[12.11cm]{
\parbox[htc]{9.7cm}{
Let $B_0=A_0^{-1},$ for $i>0$ and
$g\in \V_i,$ define $B_i g=w_3.$
\begin{enumerate}
  \item Presmoothing : $w_1=R_i g;$
  \item Correction: $w_2=w_1+ B_{i-1}Q_{i-1}(g-A_i w_1);$
  \item Postsmoothing: $w_3=w_2+R_i^*(g-A_i w_2).$
\end{enumerate}
Set $B = B_{N+1}.$}
}
\smallskip

For simplicity, we focus on the case of exact subspace solver, i.e.,
$R_i=A_i^{-1}$ for $i=0, \cdots, N$ and for the finest level, $R_{N+1}$ is chosen as Gauss-Seidel smoother, which can be also understood as the multiplicative method with exact local solvers applied to the nodal decomposition~\cite{Xu.J1992a}.  Let $P_{p}: \mathcal V\to \mathcal V_p$ and $P_i:\V\to \V_i$ be the orthogonal projection with respect to the inner product $(\cdot,\cdot)_a$. For our special choices of smoothers, we then have
\begin{align*}
I - R_{N+1}A &= \prod _{p\in \Lambda}(I- P_p),\\
I - B_{N}A &= \left(\prod_{i=0}^N(I-P_i)\right)^*
\left(\prod_{i=0}^N(I-P_i)\right),\\
\left\|I-BA\right\|_A&=\left\|\prod_{i=0}^N(I-P_i)\prod _{p\in \Lambda}(I- P_p)\right\|_A^2.
\end{align*}
For exact local solvers, we can apply the crucial X-Z identity~\cite{Xu.J;Zikatanov.L2002} to conclude
\begin{equation}\label{eqn:xz}
\left\|I-BA\right\|_A = 1-\frac{1}{1+c_0},
\end{equation}
where
 \begin{eqnarray*}
  &&c_0=\sup_{\left\|v\right\|_A=1}\inf_{v=\sum_{p\in \Lambda}v_p + \sum _{i=0}^N v_i}\left (\sum_{i=0}^N \Big \|P_i\sum_{j=i+1}^N v_j +P_i \sum_{p\in \Lambda}v_p \Big \|^2_A 
  +\sum _{p\in \Lambda}\big \|P_p \sum _{q>p}v_q\big \|_A^2\right) .
  \end{eqnarray*}

\begin{theorem}\label{th:vp}
Given a triangulation $\T_N = \T_0 + \B$ in $\mbb T(\T_0)$, let $L=\max_{\tau\in \T_N}g(\tau)$. For the multigrid $V$-cycle preconditioner $B$, we have
  $$\kappa(BA)\lesssim c_d(L),\;\; \kappa_{m_0}(BA)\lesssim L^2.$$
  Consequently, we have the following the convergence rate estimate of the BPX preconditioned
  conjugate gradient method:
  \begin{eqnarray*}
\frac{\|u-u_k\|_A}{\|u-u_0\|_A} &\le& 2
\left(C_1c_d(L)-1\right)^{m_0}\left(\frac{C_0L-1}{C_0L+1}\right)^{k-m_0}.
  \end{eqnarray*}
\end{theorem}
\begin{proof}
Since $I-BA$ is a non-expansive operator, we conclude $\lambda_{\max}(BA)\leq 1$. Since $I-BA$ is SPD in the $A$-inner product and $\lambda_{\max}(BA)\leq 1$, we have
\begin{eqnarray*}
 \|I-BA\|_A &=&\max \{|1-\lambda _{\min}(BA)|, |1-\lambda
_{\max}(BA)|\} 
= 1-\lambda _{\min}(BA).
\end{eqnarray*}
To get an estimate on the minimum eigenvalue of $BA$, we only need to get a upper bound of the constant $c_0$ in \eqref{eqn:xz}.

To do so, for any $v\in \mathcal V$, we chose the decomposition in Theorem~\ref{th:stda}. That is,
$$
	v= \tilde{v} + \sum_{i=1}^{N} v_{i}, \mbox{ with } v_{0} = \I_{0}^{a} v,\;\; v_{i} = (\I_{i}^{a} - \I_{i-1}^{a}) v,
$$
where $\tilde{v} = v - \I_{N}^{a} v = \sum_{p\in \Lambda} v_{p}.$
Then by shape regularity of the triangulation, we have
$$
	c_{0} \lesssim \sum_{i=0}^N \Big \|P_i\sum_{j=i+1}^N v_j \Big \|_A^2+ \sum _{i=0}^N\left \|P_i \tilde v\right \|_A^2+ \sum _{p\in \Lambda}\left \|P_p \sum _{q>p}v_q\right \|_A^2.
$$
We estimate these three terms as follows. For the last term, by the finite overlapping of nodal bases, we have
\begin{eqnarray*}
\sum _{p\in \Lambda}\big \|P_p \sum _{q>p}v_q\big \|_A^2 &\lesssim& \sum _{p\in \Lambda}\big \|\sum _{q>p}v_q\big \|_{A,\omega_{p}}^2\\
&\lesssim& \sum _{p\in \Lambda}\|v_p\|_{A,\omega_{p}}^2\lesssim \sum _{p\in \Lambda}h_p^{-2}\|v_p\|_{0,a,\omega_{p}}^2\\
&\lesssim& \|h^{-1}(v- \I_{N}^{a} v)\|_{0,a}^{2}\lesssim \|v\|_A^2.
\end{eqnarray*}

For the middle term, we regroup by generations and use \eqref{eq:sumtilde} to get
\begin{eqnarray*}
\sum _{i=0}^N\Big \|P_i \tilde v\Big \|_A^2 &=& \sum _{k=0}^L\sum _{l, g_l =k}\Big \|P_l \tilde v\Big \|_A^2\leq \sum _{k=0}^L\sum _{l, g_l =k}\|\tilde v\|_{A,\tilde \omega _l}^2\\
&\lesssim&  \sum _{k=0}^L\|\tilde v\|_{A}^2 = L \|\tilde v\|_{A}^2.
\end{eqnarray*}

For the first term, we define $u_i=P_i \left(\sum_{j=i+1}^N v_j\right)$ and $u_0:= P_0 (v-v_0)$ and apply the strengthened Cauchy Schwarz inequality, cf.  Lemma \ref{lm:scs} to get
\begin{eqnarray*}
\sum_{i=0}^N \Big \|P_i\sum_{j=i+1}^N v_j \Big \|_A^2 &=&   \sum_{i=0}^N \sum_{j=i+1}^N A(u_i, v_j) \\
&\lesssim& \|v-v_0\|_A^2+\sum_{i=1}^N h_i^{-2}
  \|v_i\|_{0,a}^2\\
  &\lesssim& c_d(L)\|v\|_A^2.
\end{eqnarray*}
Here the constant $c_d(L)$ can be improved to $L^2$ if we consider the decomposition \eqref{eq:stdatv} of $v\in \widetilde{\V}.$ 
Combined with the Mini-Max Theorem \ref{th:minimax}, yields
  $$\lambda _{\min}(BA)\gtrsim c_d(L), \;\; \lambda_{m_0+1}(BA)\gtrsim L^{-2},$$ and thus 
  $$
\kappa(BA)\lesssim c_d(L),  \quad  \kappa_{m_0}(BA)\lesssim L^2.
  $$
  Finally, the convergence rate of the PCG method follows by Theorem \ref{th:pcg}. $\Box$
\end{proof}

Follow the same proof as Theorem \ref{th:vp}, we can also obtain the following convergence result for the local multigrid $V$-cycle solver.
\begin{corollary}\label{cor:vc}
  For the multigrid $V$-cycle algorithm defined above on bisection grids, we have
  $$\|E\|_A=\|I-BA\|_A=1-\frac{1}{1+c_0},$$ where $c_0\lesssim c_d(L).$
\end{corollary}
This corollary implies that multigrid alone is not robust, especially in 3D. In this case, the convergence rate of multigrid will be proportional to $1-2^{-L} \simeq 1-h^{-1}_{\min},$ which deteriorates rapidly as the mesh size become small.  Remark \ref{rk:wst} is also applicable here, i.e., all the above estimates
are estimates for the worst case. For the special circumstances mentioned in
Remark \ref{rk:wst}, the estimates can be improved in the same way. 

\section{Numerical Experiments}
\label{sec:num}
In this section, we present some numerical experiments to support the theoretical results in previous sections. In the implementation of the adaptive loop, we use a modification of the error indicator presented in~\cite{Petzoldt.M2002}. Some other a posteriori error indicators for jump coefficients problem \eqref{eq:model} can be found in~\cite{Bernardi.C;Verfurth.R2000,Chen.Z;Dai.S2002,Vohralik.M2008,Cai.Z;Zhang.S2009}. The adaptive algorithm using different error indicators will generate different grids. However, we  emphasize that the robustness of the local adaptive multilevel preconditioners is independent of how the grids are generated in the refinement procedure. 

The implementation of the BPX preconditioner and the multigrid methods are standard, and can be found in, for example,~\cite{Briggs.W;Henson.V;McCormick.S2000,Xu.J1997b}. The implementation of the PCG algorithm can be found in~\cite{Golub.G;Van-Loan.C1996,Saad.Y2003}. All numerical examples are implemented by using $i$FEM~\cite{Chen.L2008}. We only present three-dimensional examples here and refer to~\cite{Chen.L;Zhang.C2007} for two-dimensional ones. 
In the PCG algorithm, we use the stopping criterion 
$$
\frac{\|u^k-u^{k-1}\|_A}{\|u^k\|_A} \leq 10^{-10}.
$$ 

In the implementation of the local multilevel preconditioners, we use an algorithm for coarsening bisection grids introduced by \cite{Chen.L;Zhang.C2007} for two dimensional case and~\cite{Chen.L2008} for three dimensional one. The coarsening algorithm will find all compatible bisections and regroup them, with possibly different generations, into groups $\cup _{l=1}^{L'} G(l)= \{1, 2, \cdots, N\}$ such that for any $i,j\in G(l), \omega_{j} \cap \omega_{i} =\varnothing.$ 
Each coarsening step is corresponding to a \emph{level} in the multilevel terminology, and the total number of levels is $L'.$
There are two major benefits of using this coarsening algorithm.
\begin{enumerate}
	\item We do not need to store the complex bisection tree structure of the refinement procedure explicitly in the algorithm. Instead, we only need the grid information on the finest level and the coarsening subroutine will restore multilevel structure.
	\item  Our numerical evidence shows that the number of nodes will decrease around one half in one coarsening step. Therefore the constant $L'$ is much smaller than the maximal generation $L\eqsim |\log h_{\min}|$.
\end{enumerate}
In what follows, we will use some shorthand notation for the different algorithms implemented.
\begin{itemize}
	\item TPSMG stands for the $V$-cycle multigrid with \emph{Three-Point Smoothing} (TPS), which only performs smoothing on new vertices and their two direct neighbors sharing the same edge.
	\item TPSMGCG is the PCG algorithm using the TPSMG as preconditioner.
	\item TPSBPXCG  is the additive version of TPSMG preconditioner.
\end{itemize}
Among all these algorithms, the main focus of this paper is the behavior of TPSMGCG and TPSBPXCG. In the numerical experiments below, we also report some results for TPSMG  for comparison. 

Inspired by~\cite{Oswald.P1999c,Xu.J1991,Xu.J;Zhu.Y2008}, we consider solving the model equation \eqref{eq:model} in the cubic domain $\Omega = (-1,1)^3.$ Let the coefficient
$a(x)$ be the constants $a_1=a_2=1$ and $a_3=\varepsilon$ on the three regions $\Omega_1,\;\Omega_2$ and $\Omega_3$ respectively (see Figure~\ref{fig:domain}), where
$$
  \Omega_1=(-0.5, 0)^3,\Omega_2=(0,0.5)^3
\; \hbox{ and }\; \Omega_3=\Omega\setminus (\overline{\Omega}_1\cup\overline{\Omega}_2).
$$ 
\begin{figure}[htbp]
    \begin{center}
        \includegraphics[scale=0.8]{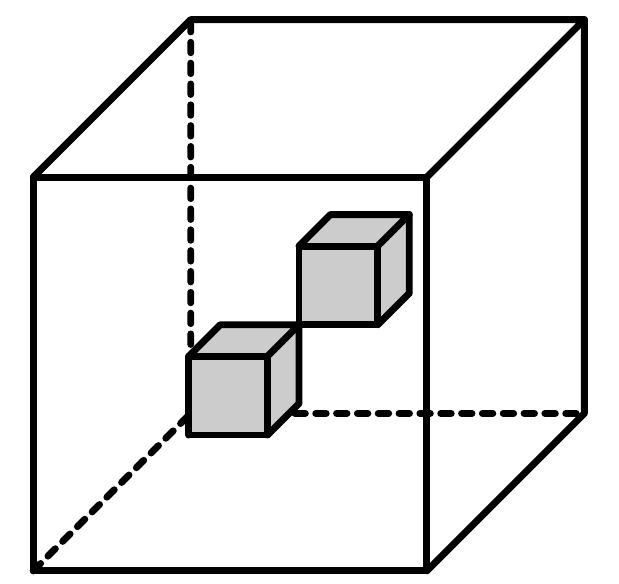}
    \end{center}
    \caption{\it The coefficients $a_{1} =a_{2}=1$ in the gray domains $\Omega _1$ and $\Omega _2,$ and $a_{3} = \varepsilon$ in the rest of the domain.}\label{fig:domain}
\end{figure}
We choose $f=1$ and impose the following boundary conditions: Dirichlet conditions $$u_{\{-1\}\times[-1,1]\times[-1,1]}=0, \quad u_{\{1\}\times[-1,1]\times[-1,1]} =1,$$ and homogenous Neumann boundary conditions on the remaining boundary.  For this problem, singularities occur along edges of $\Omega _1$ and $\Omega _2$. Figure \ref{fig:ex2-j10} shows an adaptive mesh and the corresponding finite element approximation after several iterations of the adaptive algorithm. To view the mesh around the singularity, we only show half of the domain $\Omega.$ 


\begin{figure}[htp]
\centering \includegraphics[height=5.2cm]{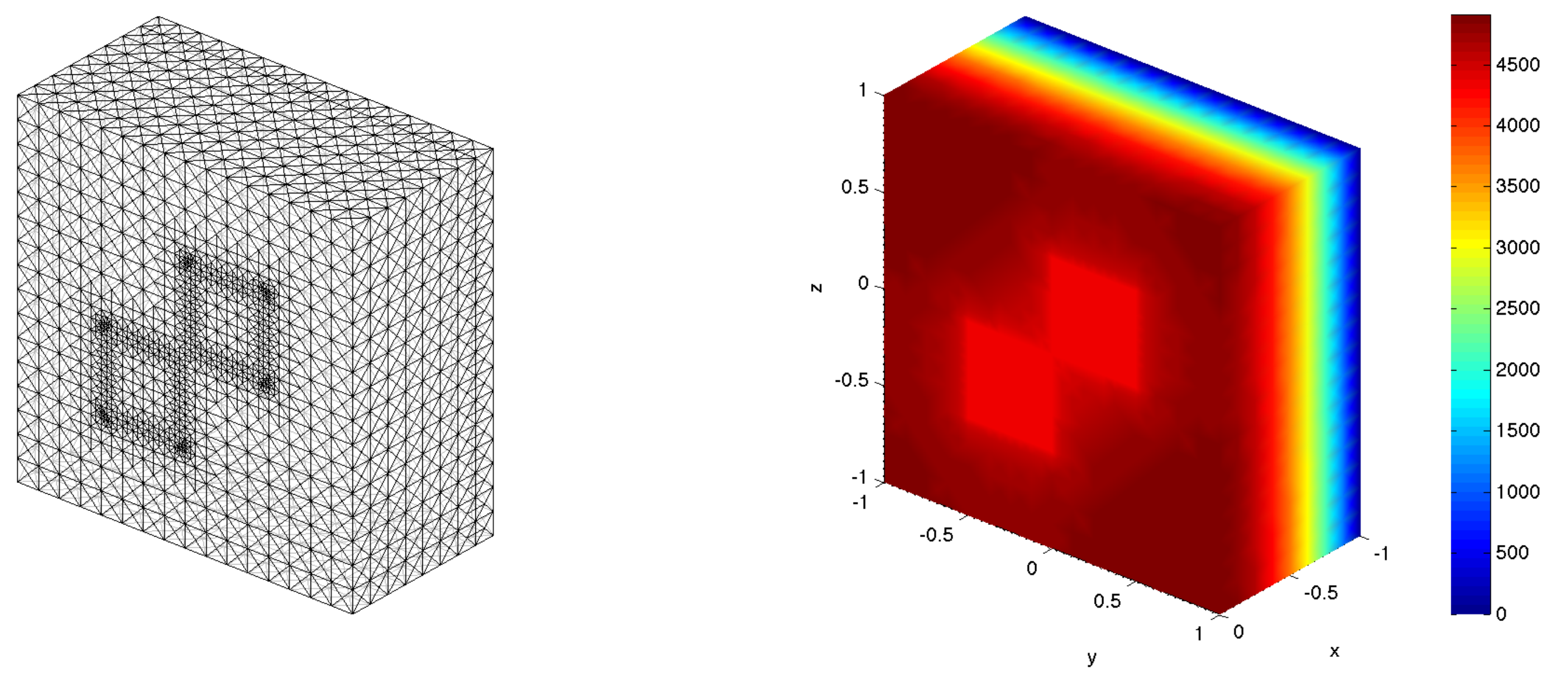}
\caption{An adaptive mesh and finite element solution with $\varepsilon = 10^{-4} $ and  $36466$ vertices.}\label{fig:ex2-j10}
\end{figure}

Tables~\ref{tab:Ex2TPSn4}-~\ref{tab:Ex2TPS4} give comparisons of the number of iterations for three different algorithms: TPSMG,  TPSMGCG  and TPSBPXCG algorithms, respectively, with the choice of $\varepsilon= 10^{-4}, 10^{-2}, 10^{2}$ and $10^{4}$. As we observe from these tables, the number of iterations for TPSMG algorithm grows rapidly as the mesh is refined when $\varepsilon$ is small. On the other hand,  the number of iterations for TPSMGCG and TPSBPXCG is very robust and only grows a little bit when the mesh is refined, as we expected from the theory.  We also observe that if $\varepsilon$ is large, the TPSMG algorithm will converge uniformly. This is because the coefficient in $\Omega_{3}$, which contains the Dirichlet boundary, is dominant. In this case, we could use the standard multigrid analysis (as in \cite{Xu.J1992a}) to show the robustness of the preconditioners.
\begin{table}[htdp]
\begin{center}
\begin{tabular}{|c||c|c|c|} 
 \hline
 DOF  & TPSMG  & TPSMGCG & TPSBPXCG \\
 \hline \hline
4913 &  41& 13 & 18\\ 
 \hline 
5505 & 62& 15 & 18\\ 
 \hline 
6617 & 89 & 18 & 21\\ 
 \hline 
8666 & 99& 19 & 19\\ 
 \hline 
10585 & 98 &19  &20\\ 
 \hline 
12411 & 125& 23 & 25\\ 
 \hline 
16353 & 154& 23 & 23\\ 
 \hline 
21248 & 182& 22 & 23\\ 
 \hline 
27755 & 197& 26 & 32\\ 
 \hline 
36466 & 178& 27 & 29\\ 
 \hline 
43271 & 238& 25& 30\\ 
 \hline 
51163 & 283& 28 & 36\\ 
 \hline 
72349 & 395& 32 & 34\\ 
 \hline 
89146 & 424& 31& 34\\ 
 \hline 
104747 & 413& 34 & 38\\ 
 \hline 
\end{tabular} 
\end{center}
\caption{Comparison of Number of Iterations for TPSMG,  TPSMGCG and TPSBPXCG when $\varepsilon = 10^{-4}$.}
\label{tab:Ex2TPSn4}
\end{table}%

\begin{table}[htdp]
\begin{center}
\begin{tabular}{|c||c|c|c|} 
 \hline
 DOF  & TPSMG  & TPSMGCG & TPSBPXCG \\
 \hline \hline
 4913 & 46&  13 & 17\\ 
 \hline 
5550 & 51& 15 &17\\ 
 \hline 
6743 & 61& 17 & 20\\ 
 \hline 
8907 & 65& 16 &19\\ 
 \hline 
10729 & 66& 17 & 20\\ 
 \hline 
13281 & 86& 20  &24\\ 
 \hline 
17146 & 90& 20 & 21\\ 
 \hline 
23139 & 90& 20 & 24\\ 
 \hline 
28613 & 160& 25& 29\\ 
 \hline 
37338 & 175& 24 & 27\\ 
 \hline 
43610 & 149& 22 & 26\\ 
 \hline 
52715 & 154& 25 & 31\\ 
 \hline 
72967 & 238& 28 & 29\\ 
 \hline 
89320 & 165& 25&33\\ 
 \hline 
113131 & 294& 30 & 38\\ 
 \hline 
\end{tabular} 
\end{center}
\caption{Comparison of Number of Iterations for TPSMG,  TPSMGCG and TPSBPXCG when $\varepsilon = 10^{-2}$.}
\label{tab:Ex2TPSn2}
\end{table}%

\begin{table}[htdp]
\begin{center}
\begin{tabular}{|c||c|c|c|} 
 \hline
 DOF  & TPSMG  & TPSMGCG & TPSBPXCG \\
 \hline \hline
4913 & 16& 10 & 14\\ 
 \hline 
5279 & 37& 15 & 15\\ 
 \hline 
5867 & 43& 17 & 18\\ 
 \hline 
6522 & 48& 16& 19\\ 
 \hline 
7562 & 68& 17& 18\\ 
 \hline 
9493 & 61& 17& 18\\ 
 \hline 
11858 & 49& 15& 18\\ 
 \hline 
15257 & 68& 15&18\\ 
 \hline 
20649 & 61& 16&19\\ 
 \hline 
27946 & 49& 17& 21\\ 
 \hline 
36735 & 52& 16& 20\\ 
 \hline 
48890 & 58& 16 & 22\\ 
 \hline 
68297 & 71& 18& 22\\ 
 \hline 
89872 & 55& 16 & 21\\ 
 \hline 
119109 & 61& 17& 23\\ 
 \hline 
\end{tabular} 
\end{center}
\caption{Comparison of Number of Iterations for TPSMG,  TPSMGCG and TPSBPXCG when $\varepsilon = 10^{2}$.}
\label{tab:Ex2TPS2}
\end{table}%

\begin{table}[htdp]
\begin{center}
\begin{tabular}{|c||c|c|c|} 
 \hline
 DOF  & TPSMG  & TPSMGCG & TPSBPXCG \\
 \hline \hline
4913 & 16& 10 & 14\\ 
 \hline 
5269 & 37&  15 & 15\\ 
 \hline 
5863 & 42& 17& 18\\ 
 \hline 
6493 & 45& 16& 18\\ 
 \hline 
7531 & 68& 17& 18\\ 
 \hline 
9419 & 59& 16 & 17\\ 
 \hline 
11721 & 46& 15& 18\\ 
 \hline 
14941 & 69& 15& 18\\ 
 \hline 
20065 & 59& 16& 19\\ 
 \hline 
27199 & 47& 17 & 21\\ 
 \hline 
35601 & 59& 16 & 20\\ 
 \hline 
47743 & 55& 16& 22\\ 
 \hline 
66989 & 71& 18& 21\\ 
 \hline 
88079 & 57& 16& 21\\ 
 \hline 
116739 & 56& 17& 23\\ 
 \hline 
\end{tabular} 
\end{center}
\caption{Comparison of Number of Iterations for TPSMG,  TPSMGCG and TPSBPXCG when $\varepsilon = 10^{4}$.}
\label{tab:Ex2TPS4}
\end{table}%

Figure \ref{fig:ex2-eigenvalues} shows the eigenvalue distributions for the TPSMGCG and TPSBPXCG preconditioned systems. As we can see from the figure, there is one small eigenvalue for both preconditioned systems. This agrees with the theoretical results,  the number of small eigenvalues is bounded by the number of floating subdomains $m_{0}\equiv 2.$ 
\begin{figure}[htbp]
\subfigure[Eigenvalues for TPSBPXCG]{
\centering
\includegraphics[height=5.7cm]{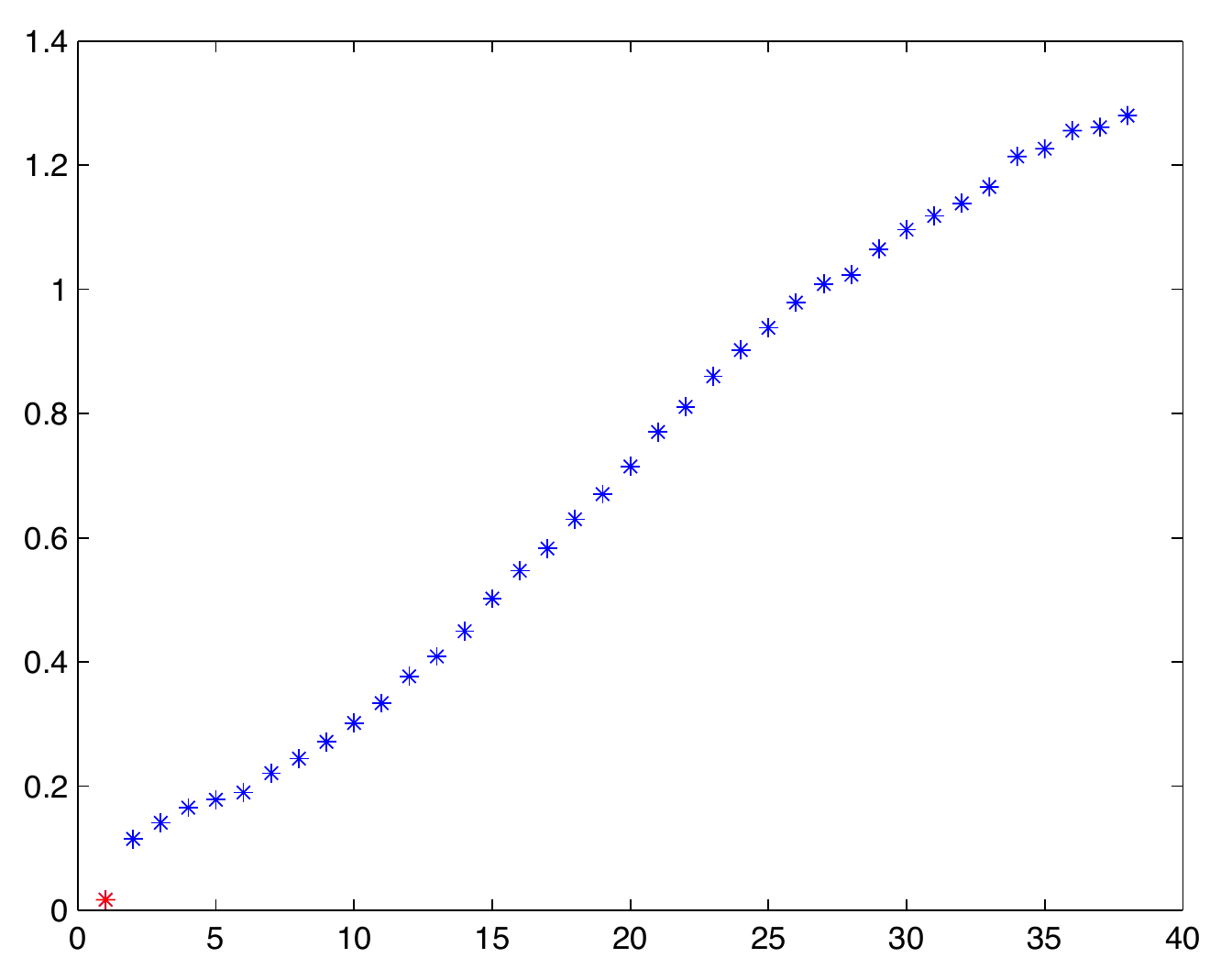}
}
\subfigure[Eigenvalues for TPSMGCG]{
\centering
\includegraphics*[height=5.7cm]{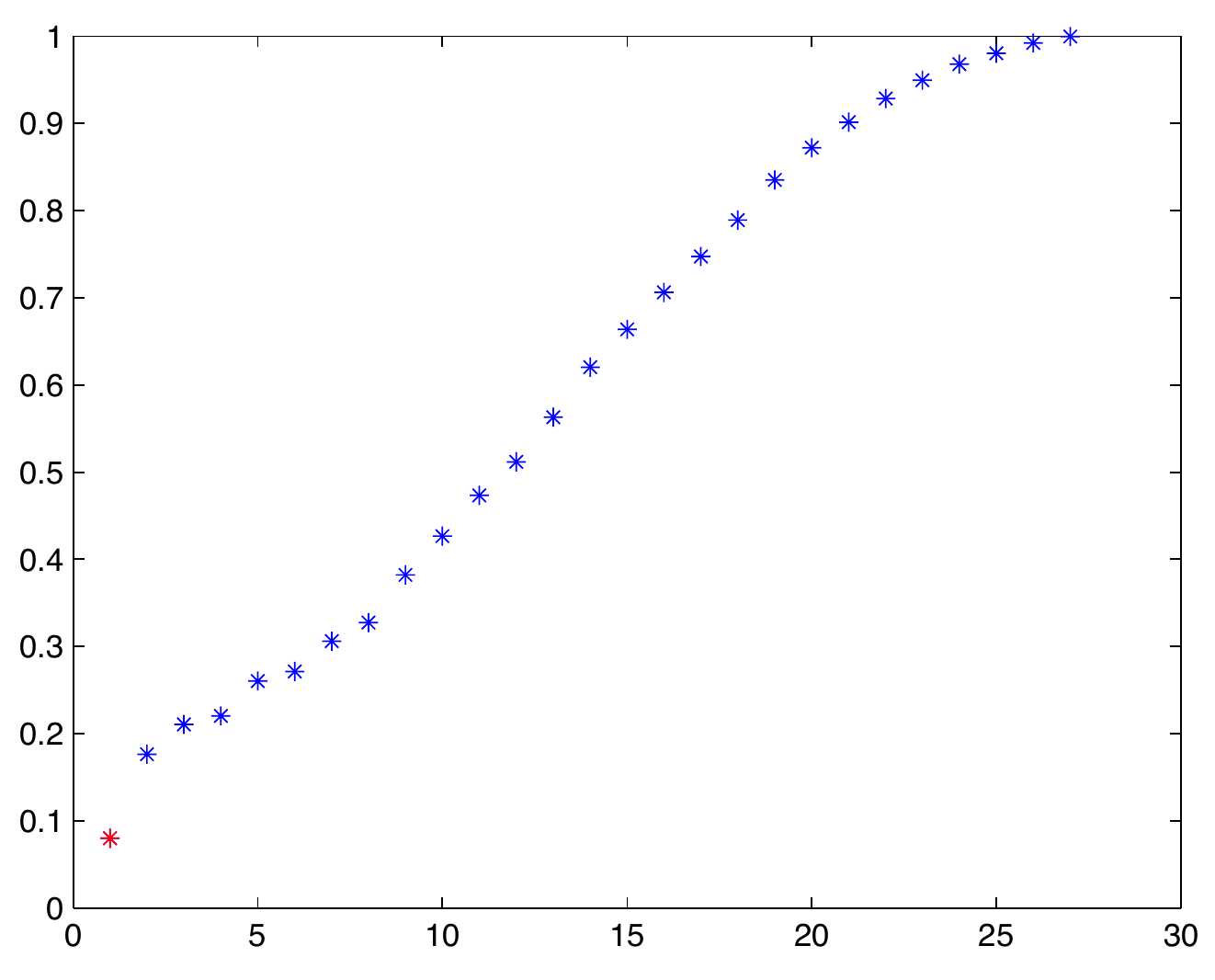}
}
\caption{Example 2: Eigenvalues of $BA$ when $\varepsilon = 10^{-4}$ with $12411$  vertices}
\label{fig:ex2-eigenvalues}
\end{figure}

Figure~\ref{fig:ex2-cond} shows the condition number and effective condition number of TPSBPXCG  and TPSMGCG preconditioned systems. From Figure~\ref{fig:ex2-cond}, we observed that when $\varepsilon$ is small, the condition number deteriorates ($\kappa(BA)\in [3, 1100]$ for TPSBPXCG, and $\kappa(BA) \in [3, 125]$ for TPSMGCG as we can see from the figure).  
On the other hand, if we get rid of the first small eigenvalue, the effective condition number $\kappa_{1}(BA)$ of TPSBPXCG  and TPSMGCG preconditioned systems (the black and red lines, respectively) are almost identical for different $\varepsilon.$ This indicates that the effective condition numbers are uniform with respect to the jumps. Moreover, as we can see from Figure~\ref{fig:ex2-cond}, $\kappa_{1}(BA)$ are mildly increasing with respect to the DOFs ($\kappa_{1}(BA)\in [1,80]$ for TPSBPXCG, and $\kappa_{1}(BA) \in [1,30]$ for TPSMGCG).
These results agree with our theoretical expectations from Section~\ref{sec:precond}.
\begin{figure}[htbp]
\subfigure[TPSBPXCG]{
\centering
\includegraphics[height=5.7cm]{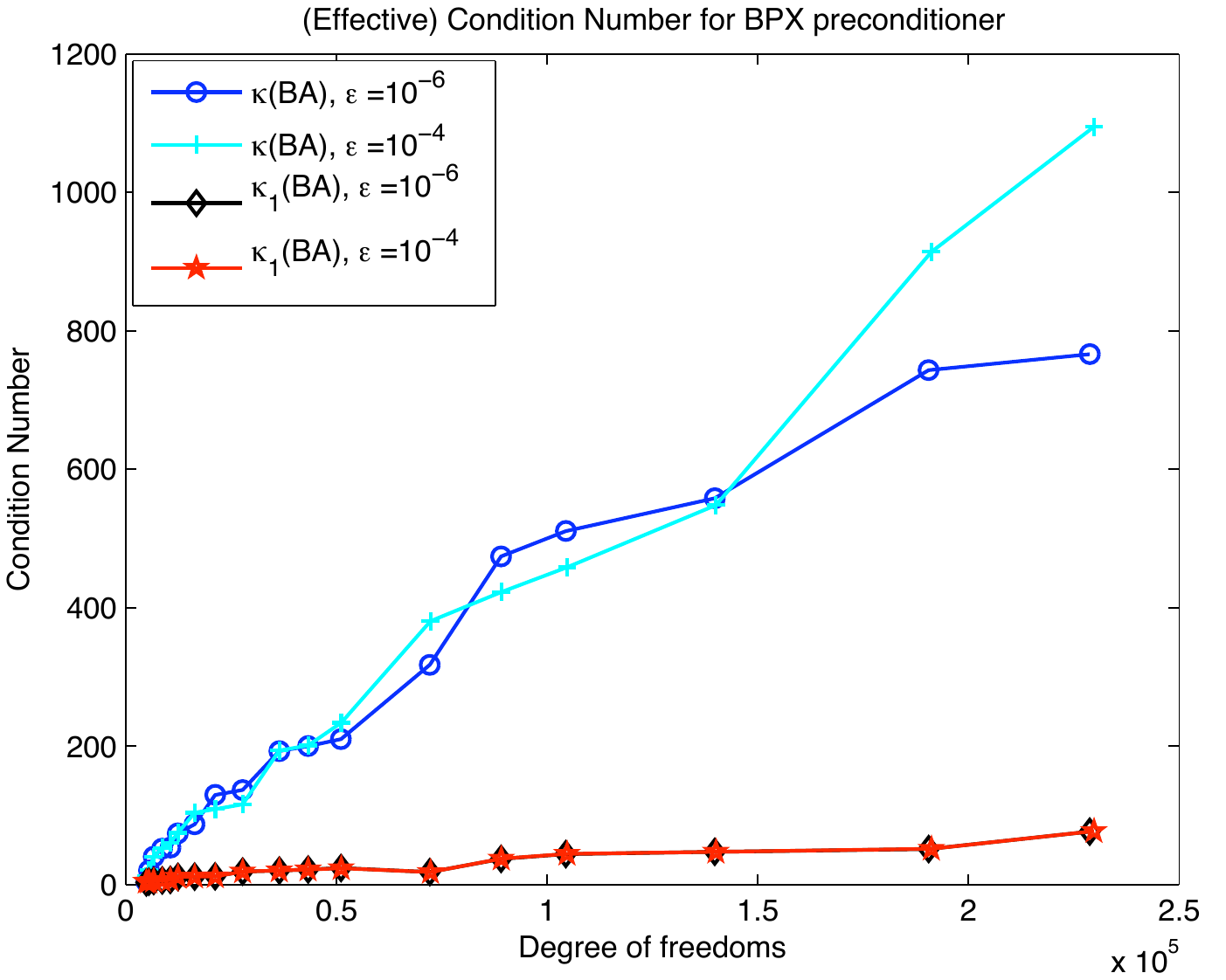}
}
\subfigure[TPSMGCG]
{
\centering
\includegraphics*[height=5.7cm]{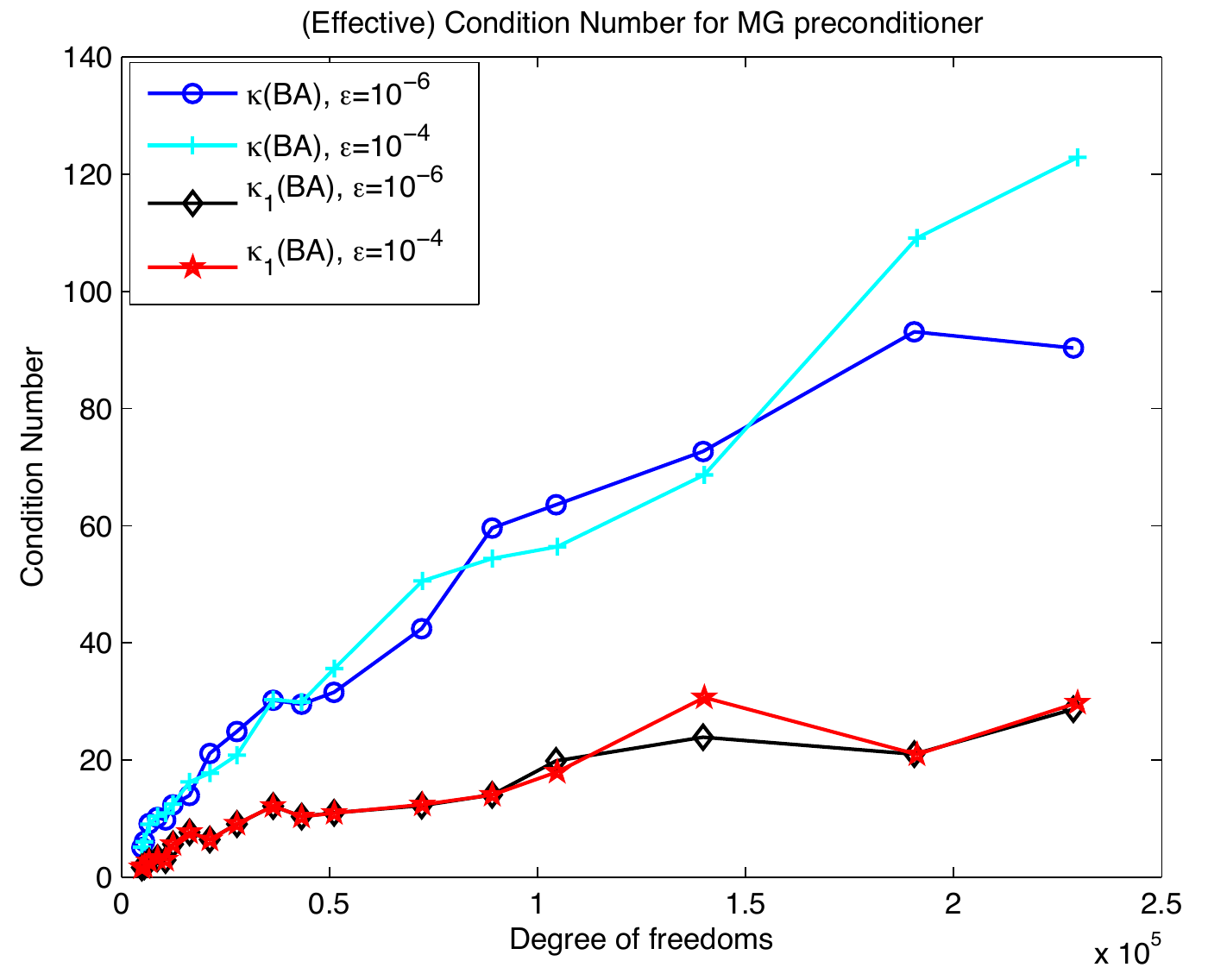}
}
\caption{Example 2: $\kappa(BA)$ and $\kappa_{1}(BA)$ for the cases $\varepsilon =10^{-6}, 10^{-4}$ w.r.t the DOFs.}
\label{fig:ex2-cond}
\end{figure}

\section{Conclusion}
In this paper, we designed local multilevel preconditioners based on the decomposition of the finite element space into 3-point subspaces for the highly graded mesh obtained from adaptive bisection algorithms. To analyze the behavior of the local multilevel preconditioners, we introduced a local interpolation operator and proved some approximation and stability properties of it. Based on these properties, we showed the decomposition of the finite element space is stable, which is a key ingredient  in the multilevel analysis. This enabled us to analyze the eigenvalue distributions of the preconditioned systems. In particular, we showed that there are only a small fixed number of eigenvalues that are deteriorated by the coefficients and mesh size, and the other eigenvalues are uniformly bounded with respect to the coefficients and logarithmically depends on the mesh size. As a result, we proved the asymptotic convergence rate of the PCG algorithm is uniform with respect to the coefficient and nearly uniform with respect to the mesh size. Moreover, the overall computation complexity of these multilevel preconditioner are nearly optimal. Numerical experiments justified our theoretical results.

\section*{Acknowledgement}
The first author is supported in part by NSF Grant DMS-0811272, NIH Grant P50GM76516 and R01GM75309. This work is also partially supported by the Beijing International Center for Mathematical Research. The second and fourth authors were supported in part by
   NSF Awards 0715146 and 0915220, and DTRA Award HDTRA-09-1-0036.  The third author was 
supported in part by NSF DMS-0609727, NSFC-10528102 and Alexander von
Humboldt Research Award for Senior US Scientists. 

\end{document}